\documentclass[reqno]{amsart}

\usepackage{hyperref}
\usepackage{amsmath}
\usepackage{amssymb}
\usepackage{amsfonts}
\usepackage{amsthm}
\usepackage{latexsym}
\usepackage{esint}
\usepackage{mathtools}
\usepackage{mathrsfs}
\usepackage{graphicx}
\usepackage{bbm}
\usepackage{color}
\usepackage{bm}
\usepackage[top=1in, bottom=1.25in, left=1.10in, right=1.10in]{geometry}
\usepackage{todonotes}

\newtheorem{theorem}{Theorem}[section]
\newtheorem{lemma}[theorem]{Lemma}

\newtheorem{corollary}[theorem]{Corollary}

\theoremstyle{definition}
\newtheorem{definition}[theorem]{Definition}

\theoremstyle{remark}
\newtheorem{remark}[theorem]{Remark}

\numberwithin{equation}{section}

\newcommand{\bx}{\mathbf{x}}
\newcommand{\bq}{\mathbf{q}}
\newcommand{\bu}{\mathbf{u}}
\newcommand{\bfu}{\mathbf{u}}

\newcommand{\bff}{\mathbf{f}}
\newcommand{\vu}{\mathbf{u}}

\newcommand{\bfpsi}{\boldsymbol{\psi}}

\newcommand{\vc}[1]{{\bf #1}}

\newcommand{\dq}{\, \mathrm{d} \mathbf{q}}
\newcommand{\dx}{\, \mathrm{d} \mathbf{x}}
\newcommand{\dt}{\, \mathrm{d}t}
\newcommand{\ds}{\, \mathrm{d}\sigma}
\newcommand{\dif}{\, \mathrm{d}}
\newcommand{\D}{\, \mathrm{d}}

\newcommand{\Q}{\mathcal{Q}}

\newcommand{\Div}{\mathrm{div}_{\mathbf{x}}}
\newcommand{\divx}{\mathrm{div}_{\mathbf{x}}}
\newcommand{\divq}{\mathrm{div}_{\mathbf{q}}}

\newcommand{\nabx}{\nabla_{\mathbf{x}}}
\newcommand{\nabq}{\nabla_{\mathbf{q}}}

\newcommand{\Delx}{\Delta_{\mathbf{x}}}
\newcommand{\Grad}{\nabla_{\mathbf{x}}}

\newcommand{\R}{\mathbb{R}}
\newcommand{\mr}{\mathbb{R}}
\newcommand{\mt}{\mathcal{T}^d}
\newcommand{\tor}{\mathcal{T}^d}

\newcommand{\inttO}{\int_{\Q_t}}

\begin{document}

\title
{Local well-posedness of the compressible FENE dumbbell model of Warner type}

\author{Dominic Breit}
\address{Department of Mathematics, Heriot-Watt University, Edinburgh, EH14 4AS, United Kingdom}
\email{d.breit@hw.ac.uk}


\author{Prince Romeo Mensah}
\address{Gran Sasso Science Institute, Viale F. Crispi, 7. 67100 L'Aquila, Italy}
\email{romeo.mensah@gssi.it}


\date{\today}



\begin{abstract}
We consider a dilute suspension of dumbbells joined by a finitely extendible nonlinear elastic (FENE) connector evolving under the classical  Warner potential $U(s)=-\frac{b}{2} \log(1-\frac{2s}{b})$, $s\in[0,\frac{b}{2})$. The solvent under consideration is modelled by the compressible Navier--Stokes system  defined on the torus $\mathbb{T}^d$ with $d=2,3$ coupled with the Fokker--Planck equation (Kolmogorov forward equation) for the probability density function of the dumbbell configuration. We prove the existence of a  unique local-in-time solution to the coupled system where this solution is smooth in the spacetime variables and interpreted weakly in the elongation variable.
 Our result holds true independently of whether or not the centre-of-mass diffusion term is incorporated in the Fokker--Planck equation.
\end{abstract}

\maketitle

\section{Introduction}

The interactions of polymer molecules and fluids are of great importance in many areas of applied sciences. Polymeric fluid analysis also has various practical applications including performances in industrial and household items such as paints, lubricants, plastics and in the processing of food stuff, see \cite{bird1987dynamics}.
\\
A common mathematical model to describe the behaviour of such complex fluids is the FENE dumbbell model introduced by Warner \cite{warner1972kinetic}, where the polymer molecules are
idealized as a bead-spring chain with a finitely extensible nonlinear elastic (FENE) type spring potential. Two beads are connected by a spring which is represented by a vector
$\bq\in B$, where $B\subset\R^d$, $d=2,3$ is a ball with radius $\sqrt{b}$ around the origin.\\
On the mesoscopic level, we describe the evolutionary changes in the distribution of the dumbbell configuration  by the Fokker--Planck equation for the polymer density function $\psi =\psi(t, \mathbf{x},\mathbf{q})$ (depending on time $t\geq0$, spatial position $\bx\in\R^d$ and the prolongation vector $\bq\in B$ of the spring).
On the macroscopic level, we consider a viscous isentropic fluid described by the compressible
Navier--Stokes equations for the fluid velocity $\bu=\bu(t,\bx)$ and density $\varrho=\varrho(t,\bx)$. The beads of the dumbbells, which model the monomers that join to form a polymer chain, unsettle the flow field around the dumbbells once immersed in the fluid.
These mesoscopic effects of the polymer molecules on the fluid motion are described by an elastic stress tensor $\mathbb{T}$. It is meant to describe the random movements of polymer chains/springs and can be modelled using the \textit{spring potential} $ U$. The potential $U$ is unbounded on the interval $[0,b/2)$,  $ U (0)=0$ and belongs to the class $ U \in C^1\big([0,b/2); \mathbb{R}_{\geq0}\big)$.  The \textit{elastic spring force} $\mathbf{F}:B \subset \mathbb{R}^d \rightarrow \mathbb{R}^d$ and associated \textit{Maxwellian} $M$ are defined by
\begin{align}
\label{elasticSpringForce}
\mathbf{F} (\mathbf{q}) = \nabq U\Big(\frac{1}{2}\vert \mathbf{q} \vert^2 \Big) = U'\Big(\frac{1}{2}\vert \mathbf{q} \vert^2 \Big) \mathbf{q}
\end{align}
and
\begin{align}
\label{maxwellian}
M(\mathbf{q}) = \frac{e^{-U \left(\frac{1}{2}\vert \mathbf{q} \vert^2 \right) }}{\int_Be^{-U \left(\frac{1}{2}\vert \mathbf{q} \vert^2 \right) }\,\mathrm{d}\mathbf{q}}
\end{align}
respectively such that $\int_B M(\mathbf{q})\,\mathrm{d}\mathbf{q}=1$. 
Several of such models are proposed in the literature, see for example \cite[Table 11.5-1]{bird1987dynamics}. We will concentrate our attention on the following \textit{spring potential} introduced by Warner \cite{warner1972kinetic}
\begin{align*}
 U (s) = -\frac{b}{2} \log \bigg(1-  \frac{2s}{b} \bigg), \quad s\in [0,b/2), \quad b>2,
\end{align*}
or equivalently, given \eqref{elasticSpringForce}, the following \textit{elastic spring force}
\begin{align}
\mathbf{F} (\mathbf{q} ) = \frac{b\,\mathbf{q}}{b-\vert \mathbf{q}\vert^2}
\end{align}
for $\mathbf{q}\in B$ so that in particular, $  \vert \mathbf{q} \vert^2 \neq b$. The choice of the above potential or spring force reflects its conformity with physical applications unlike other unrealistic models such as the Hookean dumbbell and Hookean bead-spring models which assume arbitrary large extensions of their polymer chains, see for instance \cite{chemin2001about, lions2000global, renardy1991existence}.\\
For simplicity of the presentation, we focus on periodic boundary conditions in space, hence
the underlying domain can be identified with the flat torus $\mt\subset\R^d$. If $\psi =\psi(t, \mathbf{x},\mathbf{q})$ is the \textit{probability density function} of the polymer depending on time $t\geq0$, spatial position $ \mathbf{x}\in\tor$
and of the \textit{elongation /conformation vector} $\mathbf{q} \in \mathbb{R}^{d}$, then the elastic stress tensor $\mathbb{T}$ is given by
\begin{align}
\label{extraStreeTensor}
\mathbb{T}(\psi) 
=
 \int_B \psi(t, \mathbf{x},\mathbf{q})U'\Big(\frac{1}{2}\vert \mathbf{q} \vert^2 \Big) \mathbf{q} \otimes\mathbf{q} \dq.
\end{align}
Using \eqref{elasticSpringForce}, we can rewrite \eqref{extraStreeTensor} as
\begin{align}
\label{extraStreeTensor1}
\mathbb{T}(\psi) = 
 \int_B \psi(t, \mathbf{x},\mathbf{q})\, \mathbf{F}(\mathbf{q}) \otimes\mathbf{q} \dq.
\end{align}
The stress tensor \eqref{extraStreeTensor} or \eqref{extraStreeTensor1} encodes the relationship between the rheological behaviour and fluid dynamics. In particular, it elucidates how  the polymers - described by the force law for the spring - are transmitted through the fluid.
\\
In addition to the surface forces induced by
the viscous stress tensor (whose precise definition we shall introduce shortly) as well as the elastic stress tensor $\mathbb T$, we consider an external volume force  $\mathbf{f}:(t, \mathbf{x} )\in \mathbb{R}_{\geq 0} \times \tor \mapsto  \mathbf{f}(t, \mathbf{x}) \in \mathbb{R}^d$ in the fluid motion. This force may account for the influence of gravity and/or electric force as well as artificial forces produced by for example, an ultracentrifuge. \\
The coupled system is now described by the compressible Navier--Stokes--Fokker-Planck system.
We wish to find the fluid's density $\varrho:(t, \mathbf{x})\in\mathbb{R}_{\geq 0}\times \tor \mapsto  \varrho(t, \mathbf{x}) \in \mathbb{R}_{>0}$, the fluid's velocity field $\mathbf{u}:(t, \mathbf{x})\in \mathbb{R}_{\geq 0} \times \tor \mapsto  \mathbf{u}(t, \mathbf{x}) \in \mathbb{R}^d$ and the probability density function $\psi:(t, \mathbf{x}, \mathbf{q})\in\mathbb{R}_{\geq 0} \times \tor \times B \mapsto \psi(t, \mathbf{x}, \mathbf{q}) \in \mathbb{R}_{\geq0}$
such that the equations
\begin{align}
\partial_t \varrho + \mathrm{div}_ \mathbf{x}(\varrho \mathbf{u})=0,
\label{contEq}
\\
\partial_t(\varrho \mathbf{u}) +\mathrm{div}_ \mathbf{x} (\varrho\mathbf{u}\otimes \mathbf{u}) + \nabx   p(\varrho) = \divx   \mathbb{S}(\nabx \mathbf{u}) +\divx   \mathbb{T}(\psi) +\varrho \mathbf{f},
\label{momEq}
\\
\partial_t \psi + \divx  (\mathbf{u} \psi) 
=
\varepsilon \Delx \psi
-
 \divq  \big( (\nabx   \mathbf{u}) \mathbf{q}\psi \big) 
 +
\frac{A_{11}}{4\lambda}  \divq  \bigg( M \nabq  \bigg(\frac{\psi}{M} \bigg)
\bigg)
\label{fokkerPlank}
\end{align}
are satisfied pointwise a.e. in $\mathbb{R}_{\geq 0} \times \tor \times B$ subject to the following initial and boundary conditions
\begin{align}
&\bigg[\frac{A_{11}}{4\lambda}    M \nabq  \frac{\psi}{M} -(\nabx   \mathbf{u}) \bq \psi
 \bigg] \cdot \frac{\bq}{\vert \bq \vert} =0
&\quad \text{on }  \mathbb{R}_{> 0}\times \tor \times \partial B,
\label{fokkerPlankBoundary}
\\
&(\varrho, \mathbf{u})\vert_{t=0} =(\varrho_0, \mathbf{u}_0)
&\quad \text{in }  \tor,
\label{initialDensityVelo}
\\
&\psi \vert_{t=0} =\psi_0 \geq 0
& \quad \text{in } \tor \times B.
\label{fokkerPlankIintial}
\end{align}
The momentum equation is complemented by
Newton's rheological law
\begin{align}
\label{visStressTensor0}
\mathbb{S}(\nabx \mathbf{u}) = \mu^S \bigg( \nabx   \mathbf{u} +\nabx  ^T \mathbf{u} -\frac{2}{d} \divx   \mathbf{u} \, \mathbb{I}  \bigg) +\mu^B \divx   \mathbf{u} \, \mathbb{I}
\end{align}
with shear viscosity $\mu^S>0$ and bulk viscosity $\mu^B\geq0$; as well as the adiabatic pressure law
\begin{align}
\label{isentropicPressure}
p(\varrho) = a\varrho^\gamma, \quad a>0, \gamma> 1.
\end{align}
The parameter $A_{11} >0$ in \eqref{fokkerPlank} is the first component of the  symmetric positive definite \textit{Rouse matrix} or \textit{connectivity matrix} $\mathbb{A}=A_{ij}$ for polymer chains, see \cite{rouse1953theory}. The Rouse matrix describes the network of monomers that combine to form a polymer chain. A detailed analysis of this matrix can be found in \cite{gordon1979riemann, nitta1999graph}. The non-dimensional parameter $\lambda > 0$ is the \textit{Deborah number} $\mathrm{De}$ and the parameter $\varepsilon\geq 0$ is the centre-of-mass diffusion coefficient. The Deborah number, introduced by Reiner \cite{reiner1964deborah}, measures the time it takes a material to be restored to an equilibrium state following a disturbance and the time it takes to observe the aforementioned material.
 Most models in the literature usually ignore the centre-of-mass diffusion term $\varepsilon\Delta_\bx \psi$ in \eqref{fokkerPlank}. However, an alternative school of thought, see \cite{barrett2007existence, degond2009kinetic, schieber2006generalized}, gives justifications for the inclusion of this term. We do not want to enter this discussion. Instead, we provide an approach which is suitable for both cases.\\
Existence of a solution to the Fokker--Planck equation for a given solenoidal velocity field incorporating the center-of-mass diffusion term has been established by El-Kareh and Leal \cite{el1989existence} independently of the Deborah number. 
The incompressible Navier--Stokes--Fokker--Planck system (when the time evolution of the fluid is described by the incompressible Navier--Stokes equations) for polymeric fluids including centre-of-mass diffusion (the case $\varepsilon>0$) has been studied considerably. See for example, the works by Barrett, Schwab \& S\"uli \cite{barrett2005existence}, Barrett \& S\"uli \cite{ barrett2007existence, barrett2008existence, barrettSuli2011existence, barrett2012existenceMMMAS, barrett2012existenceJDE}. 
All these results derive global-in-time weak solutions for variations of the incompressible Navier--Stokes equation coupled with the Fokker--Planck equation. On the other hand, a unique local-in-time strong solution for the centre-of-mass system was first shown to exist by Renardy \cite{renardy1991existence}. Unfortunately, \cite{renardy1991existence} excludes the physically relevant FENE dumbbell models. The local theory was then revisited by Jourdain, Leli\`evre \& Le Bris \cite{jourdain2004existence} for the stochastic FENE model for the simple Couette flow
and by E, Li \& Zhang \cite{e2004well} who analysed the  incompressible Navier--Stokes equation coupled with a system of SDEs describing the configuration of the spring (rather than the Fokker--Planck equation
for their probability distribution). The corresponding deterministic system (the incompressible Navier--Stokes equations coupled with the Fokker--Planck equation) was studied by Li, Zhang \& Zhang \cite{li2004local} and Zhang \& Zhang \cite{zhang2006local}. Constantin proved the existence of Lyapunov functionals and smooth solutions in \cite{constantin2005nonlinear} and then derived global-in-time strong solution for the 2-D system in \cite{constantin2007regularity} together with Fefferman, Titi \& Zarnescu. \\
If $\varepsilon=0$, the analysis is significantly harder since \eqref{fokkerPlank} becomes a degenerate parabolic equation which behaves like an hyperbolic equation in the spacetime $(t,\bx)$-variable.  A global weak solution result to the incompressible Navier--Stokes--Fokker--Planck system for the FENE dumbbell model without centre-of-mass diffusion
has recently been achieved in the seminal paper \cite{masmoudi2013global} by Masmoudi.  The main difficulty is to pass to the limit in the term $\divq  \big( (\nabx   \mathbf{u}) \mathbf{q}\psi \big) $ on the right-hand side of \eqref{fokkerPlank} which does not have any obvious compactness properties.
Earlier global weak solution results when $\varepsilon=0$ include the work by Lions \& Masmoudi \cite{lions2000global} for Oldroyd models,  Lions \& Masmoudi \cite{ lions2007global} who study the corotational case, and Otto \& Tzavaras \cite{otto2008continuity} who study weak solutions for the stationary system.
Masmoudi \cite{masmoudi2008well} also constructed a local-in-time strong solution to the incompressible Navier--Stokes--Fokker--Planck system for the FENE dumbbell model without centre-of-mass diffusion
in \cite{masmoudi2008well}. Furthermore, the solution is global near equilibrium, see also  Kreml \& Pokorn\'y \cite{klainerman1981singular}. The corresponding result of \cite{masmoudi2008well} in Besov spaces is shown by  Luo \& Yin \cite{luo2017global}.
\\
There are a few results in the compressible case. An extensive analysis in the 3D and 2D case has been performed by Barrett \& S\"uli \cite{barrett2016existenceA, barrett2016existence} respectively with constant viscosity coefficients and by Feireisl, Lu \& S\"uli \cite{feireisl2016dissipative} with variable viscosity coefficients.
However, all of them are concerned with the existence of weak solutions for the problem with centre-of-mass diffusion. Related results include the work  by Barrett \& S\"uli \cite{barrett2018existenceFENEP} for the FENE-P model and by Barrett, Lu \& S\"uli \cite{barrett2017existenceOldroyd} for the Oldroyd-B model.
We are not aware of any results on strong solutions.
Also, there are no results for the compressible system without centre-of-mass diffusion. We close both gaps in this paper and prove the existence of a unique local-in-time  solution to \eqref{contEq}--\eqref{fokkerPlank} under the assumption that $\varepsilon\geq0$ (see Theorem \ref{thm:main0} in the next section for the precise statement). This solution is classically strong in the spacetime variables and interpreted weakly in the elongation variable.
\\
The strategy of our proof works as follows. Inspired by \cite{breit2018local}, we rewrite the compressible Navier--Stokes equations - 
by dividing the momentum equation \eqref{momEq} by the density - into a symmetric hyperbolic system perturbed by partial viscosity. For a given elastic potential, we derive higher order energy estimates for the velocity and density (see Section \ref{sec:fluid}; in particular Theorems \ref{thm:main} and \ref{thm:main'}).
On the other hand, we derive estimates for the polymer density function if the density and viscosity of the fluid are given (see Section \ref{sec:FP}; in particular Theorems \ref{thm:mainFP} and \ref{thm:main'FP}).
This is significantly more complicated than the analysis of the incompressible Fokker--Planck equation 
in \cite{masmoudi2008well}. In particular, since our fluid is compressible, the solution of the Fokker--Planck equation is no longer transported by the Lagrangian flow as was the case in \cite{masmoudi2008well}. As such, it is of little use, if at all, to lift and study the Fokker--Planck equation from the Eulerian description to the Lagrangian description. We therefore solve the  Fokker--Planck equation  in its entirety in the Eulerian framework. To close the fluid and kinetic equations, we use a fixed-point argument. As is now widely known in contraction arguments, we are faced with the hitherto interesting twist where after showing boundedness of the fixed point map in the natural space, one is unable to show the contraction property in the same space. Indeed, we perform a difference estimate (for two solutions of the Fokker-Planck equation in terms of the given density and velocity of the fluid, cf. Theorem \ref{thm:main'FP}) in a weaker space with lower norms. When combined with the difference estimate for the fluid system, this suffices to obtain the solution we are looking for. 
This completes the proof when the fluid system is in its symmetric hyperbolic form. A transformation in density then yields a solution to the original coupled system.
Finally, as a consequence of the proof of our main theorem, a blow-up criterion is presented in Corollary \ref{cor:blowup}.
%

\section{Preliminaries and main result}
\label{sec:prelim}
In this section, we fix the notation, collect some preliminary materials (on function spaces) and present the main result.
\subsection{Notations}
We will primarily deal with three independent variables: $t\geq 0$, $\mathbf{x}\in \tor$ and $\mathbf{q}\in B$. Here, $\tor$ is the flat torus in $\mathbb{R}^d$ with $d\in \{2,3\}$ and  $B:=B(\bm{0},\sqrt{b})\subset \mathbb{R}^{d}$ is a bounded open ball of radius $\sqrt{b} \in \mathbb{R}_{>0}=(0,\infty)$ centred at $\bm{0} \in \mathbb{R}^{d}$. Time and space variables are represented by $t$ and $\mathbf{x}$ respectively whereas  $\mathbf{q} \in B $ is the \textit{elongation /conformation vector} of a polymer molecule. The spacetime cylinder $(0,t)\times \tor$ will sometimes be denoted as $\Q_t$. For functions $F$ and $G$ and a variable $p$, we write $F \lesssim G$ and $F \lesssim_p G$ if there exists  a generic constant $c>0$ and another such constant $c(p)>0$ which now depends on $p$ such that $F \leq c\,G$ and $F \leq c(p) G$ respectively.\\
By $L^p_\bx := L^p(\tor)$ and $W^{s,p}_\bx := W^{s,p}(\tor)$ for $1\leq p\leq\infty$ and $s\in\mathbb N$, we denote the standard Lebesgue and Sobolev spaces for functions with periodic boundary conditions.
For a separable Banach space $(X,\|\cdot\|_X)$, we denote by $L^p(0,T;X)$ the space of Bochner-measurable functions $u:(0,T)\rightarrow X$ such that $\|u\|_X\in L^p(0,T)$.
Finally, 
$C([0,T];X)$ is the set of continuous functions $u:[0,T]\rightarrow X$.
  
 \subsection{Function spaces}
Let us recall some Moser-type inequalities whose  proofs can be found in the appendix of \cite{klainerman1981singular}.
let $\alpha=(\alpha_1, \ldots, \alpha_d)$ be a $d$-tuple multi-index of nonnegative integers $\alpha_i$ such that $\vert \alpha \vert = \alpha_1 + \ldots +\alpha_d\leq s$ for a nonnegative integer $s\geq0$. 
\begin{itemize}
\item For any $u,v \in W^{s,2}(\tor) \cap L^\infty(\tor)$, we have that
\begin{align}
\Vert \partial^\alpha_\bx(uv)  \Vert_{L^2_\bx}
\lesssim_{s,d}
\Vert u \Vert_{L^\infty_\bx}
\Vert \nabx^s v \Vert_{L^2_\bx}
+
\Vert v \Vert_{L^\infty_\bx}
\Vert \nabx^s u \Vert_{L^2_\bx}.
\end{align}
\item For any $u \in W^{s,2}(\tor) ,\, \nabx u \in L^\infty(\tor)$ and any $v \in W^{s-1,2}(\tor) \cap L^\infty(\tor)$, we have that
\begin{align}
\label{commutatorEstimate}
\Vert \partial^\alpha_\bx(uv) 
-
u\partial^\alpha_\bx v \Vert_{L^2_\bx}
\lesssim_{s,d}
\Vert \nabx u \Vert_{L^\infty_\bx}
\Vert \nabx^{s-1} v \Vert_{L^2_\bx}
+
\Vert v \Vert_{L^\infty_\bx}
\Vert \nabx^s u \Vert_{L^2_\bx}.
\end{align}
\item If $1 \leq \vert \alpha \vert \leq 2$, then for any $u \in W^{s,2}(\tor) \cap C(\tor)$ and an $s$-times continuously differentiable function $F$ on an open neighbourhood of  a compact set $G =\mathrm{range}[u]$, we also have
\begin{align}
\label{commutatorEstimateContinuous}
\Vert \partial^\alpha_\bx F(u)  \Vert_{L^2_\bx}
\lesssim_{s,d}
\Vert \partial_u F \Vert_{C^{s-1}(G)}
\Vert u \Vert_{L^\infty_\bx}^{\vert \alpha \vert -1} \,
\Vert \partial_\bx^\alpha u\Vert_{L^2_\bx}.
\end{align}
\end{itemize} 
%
We now define a couple of weighted spaces for functions depending on the conformation vector. For the real-valued Maxwellian $M>0$, whose precise definition is given by \eqref{maxwellian}, and $p\geq1$, we denote by
\begin{align*}
&L^ p _M(B)  =  \big\{ f\in L^1_{\mathrm{loc}}(B) \, : \, \Vert f \Vert_{L^p_M(B)}^p < \infty \big\}, 
\quad
&H^1_M(B) =  \big\{ f\in L^1_{\mathrm{loc}}(B) \, : \, \Vert f \Vert_{H^1_M(B)}^2 < \infty \big\}
\end{align*}
the Maxwellian-weighted $L^p $ and $H^1$ spaces over $B$ with norms
\begin{align*}
\Vert f \Vert_{L^ p _M(B)}^p :=  \int_B M \Big\vert  \frac{f}{M} \Big\vert^ p  \dq  
\quad \quad \text{and} \quad \quad
\Vert f \Vert_{H^1_M(B)}^2 :=  \int_B M \Big\vert \nabq   \frac{f}{M} \Big\vert^2  \dq  
\end{align*}
respectively.
The following crucial lemma is originally due to \cite{masmoudi2008well}.
\begin{lemma}\label{lem:A1}
For every $\delta>0$, there exists $c_\delta>0$ such that 
\begin{align*}
\bigg( \int_B \frac{\vert \psi \vert}{1-\vert \mathbf{q} \vert}\dq  \bigg)^2
\leq
\delta
 \int_B M \bigg\vert \nabq  \frac{\psi }{M}\bigg\vert^2\dq  
 + c_\delta
  \int_B M \bigg\vert  \frac{\psi }{M} \bigg\vert^2\dq
\end{align*}
for all $\psi\in H^1_M(B)$.
\end{lemma}
 Finally, for $s\in\mathbb N$ and $p\geq1$ we define the spaces $W^{s,p}(\mt;L_M^2(B))$ and $W^{s,p}(\mt;H_M^1(B))$ respectively, as the set of measurable functions $f$ on $\mt\times B$ for which the corresponding norm
\begin{align*}
\|f\|^p_{W^{s,p}_{\bx}L^2_M}:= \sum_{\vert \alpha \vert\leq s}\int_{\mt}\bigg(\int_B M \Big\vert\partial^\alpha_\bx  \frac{f}{M} \Big\vert^2  \dq \bigg)^\frac{p}{2} \dx,
\quad 
\|f\|^p_{W^{s,p}_{\bx}H^1_M}:=\sum_{\vert \alpha \vert\leq s}\int_{\mt}\bigg(\int_B M \Big\vert \partial^\alpha_\bx\nabla_{\bq}  \frac{f}{M} \Big\vert^2  \dq\bigg)^\frac{p}{2} \dx 
\end{align*}
is finite.
%
%

\subsection{Main result}
We start by giving a rigorous definition of solution to the coupled system \eqref{contEq}--\eqref{fokkerPlank}.
\begin{definition}
\label{def:strsolmartRho}
Let $s\in\mathbb N$ and $T>0$. Assume that
$(\varrho_0,\vu_0)\in W^{s,2}(\mt)\times W^{s,2}(\mt;\mathbb{R}^d )$, $\psi_0\in W^{s,2} \big( \tor; L^2_M(B) \big) $
and $\bff\in C([0,T]; W^{s,2}(\tor;\mathbb{R}^d ))$.
We call the triple
$\left(\varrho,\vu,\psi\right)$
a solution to the system \eqref{contEq}--\eqref{fokkerPlank} with initial condition $(\varrho_0,\vu_0,\psi_0)$ in the interval $[0,T]$ provided the following holds.
\begin{itemize}
\item[(a)] $\varrho$ satisfies
$$\varrho \in  C([0,T]; W^{s,2}(\tor)), \ \varrho>0;$$

\item[(b)] $\vu$ satisfies
$$ \vu \in C([0,T]; W^{s,2}(\tor;\mathbb{R}^d))\cap L^2(0,T;W^{s+1,2}(\mt;\mathbb{R}^d));$$
\item[(c)] $\psi$ satisfies
\begin{align*}
\psi\in C \big( [0, T]; W^{s,2} \big( \tor; L^2_M(B) \big) \big)  
\cap L^2 \big( 0,T; W^{s,2}\big(\tor; H^1_M(B) \big)\big);
\end{align*}
\item[(d)] there holds for all $(t,\bx)\in[0,T] \times \mt$, for a.e. $\bq \in  B$ and for any $\phi(\bq) \in C^1(B)$,
\begin{equation}
\begin{aligned}
\nonumber
\varrho(t) &= \varrho_0 -  \int_0^{t} \Div (\varrho\vu) \,\mathrm{d}\sigma, 
\\
\varrho\vu (t)  &= \varrho_0\vu_0 - \int_0^{t} \Big[ \Div\left(\varrho \vu \otimes \vu\right)  +  \nabx p(\varrho) -  \Div \mathbb{S} (\Grad \vu) -  \Div \mathbb{T}(\psi)   - \varrho\bff\Big]  \,\mathrm{d}\sigma,\\
\int_B\psi (t)\phi\dq  &= \int_B\psi_0\phi \dq - \int_0^t \int_B\Big[ \divx  (\mathbf{u} \psi) 
-
\varepsilon \,\Delx \psi
 \Big] \phi\dq\dif\sigma
+
\int_0^t \int_B
 (\nabx   \mathbf{u}) \mathbf{q}\psi \cdot \nabq \phi \, \dq\dif\sigma
\\&
-
\frac{A_{11}}{4\lambda}  \int_0^t\int_B
M \nabq  \bigg(\frac{\psi}{M} \bigg)\cdot \nabq \phi \,\dq\dif\sigma.
\end{aligned}
\end{equation} 
\end{itemize}
\end{definition}
We remark that the assumed regularity of the solution $(\varrho,\vu,\psi)$ as well as the data together with the equations in (d) immediately imply that $\varrho$, $\vu$ and $\psi$ are differentiable in time so that indeed we obtain \eqref{contEq}--\eqref{fokkerPlank} (with \eqref{fokkerPlank} understood in the weak sense with respect to the elongation variable). Similarly, it is easy to see that $\varrho(0)=\varrho_0$, $\bu(0)=\bu_0$ and $\psi(0)=\psi_0$.
Our main result, stated below in Theorem \ref{thm:main0}, establishes the existence of a unique local-in-time solution to the coupled kinetic-fluid system \eqref{contEq}--\eqref{fokkerPlank} satisfying the boundary and initial conditions \eqref{initialDensityVelo}--\eqref{fokkerPlankBoundary}.
\begin{theorem}
\label{thm:main0}
Let $s\in\mathbb N$ satisfy $s>\frac{d}{2} + 2$ and suppose that $\varepsilon\geq0$. Assume that
$(\varrho_0,\vu_0)\in W^{s,2}(\mt) \times W^{s,2}(\mt;\mr^d)$, $\varrho_0>0$,  $\psi_0\in W^{s,2} \big( \tor; L^2_M(B) \big) $
and $\bff \in C([0,T_0]; W^{s,2}(\tor;\mathbb{R}^d))$. Then there is $T>0$
such that there is a unique solution $(\varrho,\vu,\psi)$ to problem \eqref{contEq}--\eqref{fokkerPlank}
in the sense of Definition \ref{def:strsolmartRho} on the interval $[0,T]$ with the initial condition $(\varrho_0,\vu_0,\psi_0)$.
\end{theorem}
\begin{remark}
One can easily adapt the method used in this paper to solve the equivalent version of Theorem \ref{thm:main0} on the whole space by imposing the following far field conditions
\begin{align*}
\varrho\rightarrow \overline{\varrho}, \quad \bu  \rightarrow 0, \quad \psi  \rightarrow M  \quad \text{as} \quad \vert \bx \vert \rightarrow \infty,
\end{align*}
where $\overline{\varrho}>0$ is a constant and $M$ is the Maxwellian \eqref{maxwellian}. An extension to bounded domains complemented with the no-slip boundary conditions for the velocity field seems more involved. In particular, a Galerkin approximation cannot be used for the justification of the estimates.
\end{remark}
We shall prove Theorem \ref{thm:main0} in an indirect way by first rewriting \eqref{contEq}--\eqref{momEq} as a symmetric hyperbolic-parabolic system in terms of $(r, \bu)$ where $r=r(\varrho)$ similar to \cite{breit2018local}. This reformulation relies on the non-anticipation of possible vacuum region (i.e. $\varrho \neq 0$) in the construction of strong solutions to \eqref{contEq}--\eqref{momEq}. Having transformed the original system \eqref{contEq}--\eqref{momEq} into a symmetric hyperbolic-parabolic one, we derive a priori estimates for $(r,\bu)$ under the assumption that the given data has enough regularity. To derive this reformulation, we first observe that formerly we have
\begin{align*}
\partial_t(\varrho \mathbf{u})= \partial_t\varrho \, \mathbf{u} + \varrho \partial_t \mathbf{u}, \quad \partial_t \varrho = - \divx   (\varrho\mathbf{u}),
\end{align*}
due to \eqref{contEq}.
As such, we can rewrite the momentum balance  equation \eqref{momEq} as
\begin{align*}
\varrho \partial_t \mathbf{u}  + \varrho\mathbf{u} \cdot  \nabx  \mathbf{u} + \nabx   p(\varrho) = \divx   \mathbb{S}(\nabx   \mathbf{u}) +\divx   \mathbb{T} +\varrho \mathbf{f}.
\end{align*}
Now since the non appearance of a vacuum state  is  anticipated for the existence of a strong solution to the compressible system, we can further rewrite the above equation as
\begin{align*}
 \partial_t \mathbf{u}  + \mathbf{u} \cdot  \nabx  \mathbf{u} + \frac{1}{\varrho}\nabx   p(\varrho) = \frac{1}{\varrho}\divx   \mathbb{S}(\nabx   \mathbf{u}) +\frac{1}{\varrho}\divx   \mathbb{T}(\psi) + \mathbf{f}.
\end{align*}
Finally, if we introduce
\begin{align}
\label{transfRhoToR}
r:=\sqrt{\frac{2a\gamma}{\gamma -1}} \varrho^{\frac{\gamma-1}{2}}, \quad 
D(r):=\frac{1}{\varrho(r)}=\bigg( \frac{\gamma -1 }{2a\gamma} \bigg)^\frac{-1}{\gamma-1} r^\frac{-2}{\gamma -1},
\end{align}
then the mass-momumtum balance equations \eqref{contEq}--\eqref{momEq} together with the Fokker--Planck equation \eqref{fokkerPlank} become
\begin{align}
\partial_t r + \mathbf{u}\cdot \nabx   r +\frac{\gamma -1}{2} r\, \divx   \mathbf{u} =0, \label{contEquR}
\\
 \partial_t \mathbf{u}  + \mathbf{u} \cdot  \nabx  \mathbf{u} + r\nabx   r = D(r)\divx   \mathbb{S}(\nabx   \mathbf{u}) +D(r)\divx   \mathbb{T}(\psi) + \mathbf{f},\label{momEquR}
 \\
\partial_t \psi + \divx  (\mathbf{u} \psi) 
=
\varepsilon \Delx \psi
-
 \divq  \big( (\nabx   \mathbf{u}) \mathbf{q}\psi \big) 
 +
\frac{A_{11}}{4\lambda}  \divq  \bigg( M \nabq  \bigg(\frac{\psi}{M} \bigg)
\bigg) \label{fokkerPlankR} 
\end{align}
respectively. Setting $r_0:=\sqrt{\frac{2a\gamma}{\gamma -1}} \varrho_0^{\frac{\gamma-1}{2}}$ we can endow the above system with the following initial and boundary conditions
\begin{align}
&\bigg[\frac{A_{11}}{4\lambda}    M \nabq  \frac{\psi}{M} -(\nabx   \mathbf{u}) \bq \psi
 \bigg] \cdot \frac{\bq}{\vert \bq \vert} =0
&\quad \text{on }  \mathbb{R}_{> 0}\times \tor \times \partial B,
\\
&(r, \mathbf{u})\vert_{t=0} =(r_0, \mathbf{u}_0)
&\quad \text{in }  \tor,
\label{initialDensityR}
\\
&\psi \vert_{t=0} =\psi_0 \geq 0
& \quad \text{in } \tor \times B.
\end{align} 
Analogous to Definition \ref{def:strsolmartRho}, we give in the following, the definition of a solution to \eqref{contEquR}--\eqref{fokkerPlankR}.
\begin{definition}
\label{def:strsolmartR}
Let $s\in\mathbb N$ and $T>0$. Assume that
$(r_0,\vu_0)\in W^{s,2}(\mt)\times W^{s,2}(\mt;\mathbb{R}^d )$, $\psi_0\in W^{s,2} \big( \tor; L^2_M(B) \big) $
and $\bff\in C([0,T]; W^{s,2}(\tor;\mathbb{R}^d ))$.
We call the triple
$\left(r,\vu,\psi\right)$
a solution to the system \eqref{contEquR}--\eqref{fokkerPlankR} with initial condition $(r_0,\vu_0,\psi_0)$ in the interval $[0,T]$ provided the following holds.
\begin{itemize}
\item[(a)] $r$ satisfies
$$r \in  C([0,T]; W^{s,2}(\tor)), \ r>0;$$

\item[(b)] the velocity $\vu$ satisfies
$$ \vu \in C([0,T]; W^{s,2}(\tor;\mathbb{R}^d))\cap L^2(0,T;W^{s+1,2}(\mt;\mathbb{R}^d));$$
\item[(c)] $\psi$ satisfies
\begin{align*}
\psi\in C \big( [0, T]; W^{s,2} \big( \tor; L^2_M(B) \big) \big)  
\cap L^2 \big( 0,T; W^{s,2}\big(\tor; H^1_M(B) \big)\big);
\end{align*}
\item[(d)] there holds for all $(t,\bx)\in[0,T] \times \mt$, for a.e. $\bq \in  B$ and for any $\phi(\bq) \in C^1(B)$,
\begin{equation}
\begin{aligned}
\nonumber
r(t) &= r_0 -  \int_0^{t}\Big[\vu \cdot \Grad r \  + \tfrac{\gamma - 1}{2}\, r\, \Div \vu\Big] \,\mathrm{d}\sigma, 
\\
\vu (t)  &= \vu_0 - \int_0^{t} \big[ \vu \cdot \Grad \vu + r \Grad r  -  D(r) \Div \mathbb{S} (\Grad \vu)  -   D(r) \Div \mathbb{T}   - \bff \big] \,\mathrm{d}\sigma,
\\
\int_B\psi (t)\phi\dq  &= \int_B\psi_0 \phi\dq - \int_0^t \int_B\Big[ \divx  (\mathbf{u} \psi) 
-
\varepsilon \,\Delx \psi
 \Big]\phi \dq\dif\sigma
+
\int_0^t \int_B
 (\nabx   \mathbf{u}) \mathbf{q}\psi \cdot \nabq \phi \, \dq\dif\sigma
\\&
-
\frac{A_{11}}{4\lambda}  \int_0^t\int_B
M \nabq  \bigg(\frac{\psi}{M} \bigg)\cdot \nabq \phi \,\dq\dif\sigma.
\end{aligned}
\end{equation} 
\end{itemize}
\end{definition}
Majority of our effort will be concentrated on proving the following theorem.
\begin{theorem}
\label{thm:aux}
Let $s\in\mathbb N$ satisfy $s>\frac{d}{2} + 2$ and suppose that $\varepsilon\geq0$. Assume that
$(r_0,\vu_0)\in W^{s,2}(\mt) \times W^{s,2}(\mt;\mr^d)$, $r_0>0$,  $\psi_0\in W^{s,2} \big( \tor; L^2_M(B) \big) $
and $\bff \in C([0,T_0]; W^{s,2}(\tor;\mathbb{R}^d))$. Then there is $T>0$
such that there is a unique solution $(r,\vu,\psi)$ to problem \eqref{contEquR}--\eqref{fokkerPlankR}
in the sense of Definition \ref{def:strsolmartR}  in the interval $[0,T]$and  with the initial condition $(r_0,\vu_0,\psi_0)$.
\end{theorem}

\section{Solving the fluid system}
\label{sec:fluid}
We start by solving the fluid system for a given elastic stress tensor and a given external force. That is, for a given $\mathbb T$ and $\bff$, we want to solve the following system
\begin{equation} \label{E3}
\partial_t r + \vu \cdot \Grad r  + \frac{\gamma - 1}{2} r \Div \vu = 0,
\end{equation}
\begin{equation} \label{E4}
\partial_t \vu +  \vu \cdot \Grad \vu + r \Grad r   = D(r) \Div \mathbb{S} (\Grad \vu) +D(r) \Div\mathbb{T} + \bff.
\end{equation}
Let us start with a precise definition of the solution.
\begin{definition}[Strong solution] \label{def:strsolmartRfixedT}
Let $s\in\mathbb N$ and $T>0$. Assume that
$(r_0,\vu_0)\in W^{s,2}(\mt)\times W^{s,2}(\mt ; \mathbb{R}^d )$, $\mathbb T\in C([0,T]; W^{s,2}(\tor;\mathbb{R}^{d\times d}))$
and $\bff\in C([0,T]; W^{s,2}(\tor ;\mathbb{R}^d ))$.
We call the tuple
$\left(r,\vu\right)$
a strong solution to the system \eqref{E3}--\eqref{E4} with initial condition $(r_0,\vu_0)$ in the interval $[0,T]$ provided the following holds.
\begin{itemize}
\item[(a)] $r$ satisfies
$$r \in  C([0,T]; W^{s,2}(\tor)), \ r>0;$$

\item[(b)] $\vu$ satisfies
$$ \vu \in C([0,T]; W^{s,2}(\tor;\mathbb{R}^d))\cap L^2(0,T;W^{s+1,2}(\mt;\mathbb{R}^d));$$

\item[(c)] there holds for all $t\in[0,T]$
\[
\begin{split}
r(t) &= r_0 -  \int_0^{t}\Big[\vu \cdot \Grad r \  + \tfrac{\gamma - 1}{2}\, r\, \Div \vu\Big] \,\mathrm{d}\sigma, \\
\vu (t)  &= \vu_0 - \int_0^{t} \left[ \vu \cdot \Grad \vu + r \Grad r \right]  \,\mathrm{d}\sigma + \int_0^{t}  D(r) \Div \mathbb{S} (\Grad \vu) \, \mathrm{d}\sigma + \int_0^{t}  D(r) \Div \mathbb{T}   \,\mathrm{d}\sigma+ \int_0^{t} \bff  \,\mathrm{d}\sigma.
\end{split}
\]
\end{itemize}
\end{definition}
We now formulate our results concerning existence and uniqueness of solutions to \eqref{E3}--\eqref{E4}. 
\begin{theorem}\label{thm:main}
Let $s\in\mathbb N$ satisfy $s>\frac{d}{2} + 2$ and $T_0>0$. Assume that
$(r_0,\vu_0)\in W^{s,2}(\mt)\times W^{s,2}(\mt;\mr^d)$, $r_0>0$, $\mathbb T\in C([0,T_0]; W^{s,2}(\tor;\mathbb{R}^{d \times d}))$
and $\bff \in C([0,T_0]; W^{s,2}(\tor ;\mathbb{R}^d ))$. Then there is $T\in(0,T_0]$
such that there is a strong solution $(r,\vu)$ to problem \eqref{E3}--\eqref{E4}
in the sense of Definition \ref{def:strsolmartRfixedT} with the initial condition $(r_0,\vu_0)$ in the interval $[0,T]$. Moreover, we have the estimate
\begin{equation}
\begin{aligned}
\label{est_thm:main}
\sup_{0<t<T}  &\big\Vert (r , \mathbf{u}   ) \big\Vert^2_{W^{s,2}_\mathbf{x}}
+ \int_{0}^T\big\Vert \mathbf{u}  \big\Vert^2_{W^{s+1,2}_\mathbf{x}}
\dt
\leq c\,  \bigg[
 \big\Vert \big(r_0, \mathbf{u}_0  \big) \big\Vert^2_{W^{s,2}_\mathbf{x}}
+ 
\int_0^T  \Big(
\Vert  \mathbf{f} \Vert_{W^{s,2}_\mathbf{x}}^2
+
\Vert   \mathbb T\Vert_{W^{s,2}_\mathbf{x}}^2
\Big)
 \, \mathrm{d}t
 \bigg]
\end{aligned}
\end{equation}
where the constant $c$ depends only on $\gamma,s$ and $d$.
\end{theorem}
The proof of Theorem \ref{thm:main} can be found in Section \ref{sec:exisNS}.
\begin{theorem}\label{thm:main'}
Let the assumptions of Theorem \ref{thm:main} be satisfied. Then the solution from Theorem \ref{thm:main} is unique. Moreover, for $s'\in\mathbb N$ with $s' \leq s-1$ we have the estimate
\begin{equation}
\begin{aligned}
\label{est:thm:main'}
 \sup_{0<t<T} &\Vert ( r^{1}(t)-r^2(t), \mathbf{u}^{1}(t)-\bu^2(t))   \Vert^2_{W^{s',2}_\mathbf{x}}
+\int_0^T \Vert \mathbf{u}^{1}-\bu^2   \Vert^2_{W^{s'+1,2}_\mathbf{x}}\, \mathrm{d}t
\\&
\leq\,c
\exp\bigg(c\int_0^T \Big( 
\Vert (r^1,r^2)  \Vert^2_{W^{s'+1,2}_\mathbf{x}}
+ \Vert (\mathbf{u}^1,\mathbf{u}^2)   \Vert_{W^{s'+2,2}_\mathbf{x}}
+
\Vert   \mathbb T^2\Vert_{W^{s',2}_\mathbf{x}}^2+1\Big)
\, \mathrm{d}t \bigg)\\
&\qquad\qquad\times
\int_0^T\Vert   \mathbb T^1-\mathbb T^2\Vert_{W^{s',2}_\mathbf{x}}^2
 \Vert r^1 \Vert_{W^{s'+1,2}_\mathbf{x}}
\, \mathrm{d}t,
\end{aligned}
\end{equation}
for some $T=T(\sup_{t}\|\bu^1\|_{s,2},\sup_{t}\|\bu^2\|_{s,2})$.
Here $(r^1,\vu^1)$ and $(r^2,\vu^2)$ are two strong solutions to \eqref{E3}--\eqref{E4}
with data $(r_0,\vu_0,\bff,\mathbb T^1)$ and $(r_0,\vu_0,\bff,\mathbb T^2)$ respectively.
\end{theorem}
The proof of Theorem \ref{thm:main'} can be found in Section \ref{sec:uniqNS}.

\subsection{The Galerkin approximation}
\label{subsec:galerkin}
We start with a regularized system which contains cut-offs in the nonlinear term to render them globally Lipschitz. It can be solved by a standard Galerkin approximation. To be precise,
we consider for $R\gg1$
the system
\begin{align} \label{E3'}
\partial_t r + \varphi_R (\|\vu \|_{2,\infty})\Big[\vu \cdot \Grad r \  + \tfrac{\gamma - 1}{2}\, r\, \Div \vu\Big] &= 0,\\
\label{E4'}
\partial_t \vu + \varphi_R(\| \vu \|_{2,\infty})\left[ \vu \cdot \Grad \vu + r \Grad r \right]  & =\varphi_R(\| \vu \|_{2,\infty}) D(r) \Div \mathbb{S} (\Grad \vu) \\& +\varphi_R(\| \vu \|_{2
,\infty}) D(r) \Div \mathbb{T}+\bff\nonumber,\\
r(0)=r_{0} ,\quad \bfu(0)&=\bfu_0,\label{approx:initial}
\end{align}
where $\varphi_R:[0,\infty)\rightarrow[0,1]$ are smooth cut-off functions satisfying
\begin{align*}
\varphi_R(y)=\begin{cases}1,\quad &0\leq y\leq R,\\
0,\quad & R+1\leq y.
\end{cases}
\end{align*}
To begin with, observe that for any $\vu \in C([0,T]; W^{2,\infty}(\mt))$, the transport equation (\ref{E3'}) admits a classical solution
$r = r[\vu]$, uniquely determined by the initial datum $r_{0}$.
 In addition, for a certain universal constant $c$, we have the estimates
\begin{equation} \label{est1}
\begin{split}
\frac{1}{R}
\exp \left( - cR t \right)&\leq
\exp \left( - cR t \right)  \inf_{\mt} r_0  \leq r(t, \cdot) \leq \exp \left(  cR t \right) \sup_{\mt} r_{0} \leq  R \exp \left(  cR t \right), \\
|\Grad r (t,\cdot) | & \leq \exp \left(  cR t \right)|\Grad r_0 |\leq
R \exp \left(  cR t \right) \ t \in [0,T].
\end{split}
\end{equation}

Next, we consider the orthonormal basis $\left\{ \bfpsi_m \right\}_{m=1}^\infty$ of the space $L^{2}(\mt)$ formed by trigonometric functions and set
\[
X_n = {\rm span} \left\{ \bfpsi_1, \dots, \bfpsi_n \right\}, \quad \mbox{with the associated projection}\ P_n : L^2 \to X_n.
\]
We look
for approximate solutions $\bu^n$ of \eqref{E4'} belonging to {$C([0,T]; X_n)$},
satisfying
\begin{equation} \label{est2}
\begin{split}
\partial_t\int_{\mt} \bu^n\cdot \bfpsi_i \dx &+ \varphi_R(\| \bu^n \|_{2,\infty})\int_{\mt} \Big[[\bu^n \cdot \Grad \bu^n + r^n \Grad r^n \Big]\cdot \bfpsi_i  \dx
\\
&=\varphi_R(\| \bu^n \|_{2,\infty}) \int_{\mt} D(r^n) \Div \mathbb{S} (\Grad \bu^n)\cdot \bfpsi_i \dx\\  &
+  \varphi_R(\| \bu^n \|_{2,\infty}) \int_{\mt} D(r^n) \Div \mathbb{T} \cdot \bfpsi_i \dx +\int_{\mt}\bff\cdot\bfpsi\dx,\quad i = 1, \dots, n.\\
\bu^n(0)&=\bu_0^n:=P_n \bfu_0.
\end{split}
\end{equation}
Here $r^n=r[\bu^n]$ is the solution to \eqref{E3'} with $\bu=\bu^n$.
As all norms on $X_n$ are equivalent, solutions of \eqref{est2} can be obtained in a standard way by means of the Banach fixed point argument.
Specifically, we have to show that the mapping
\[
\vu \mapsto \mathscr{T} \vu : X_n \to X_n,
\]
\begin{align} \nonumber
\int_{\mt} \mathscr{T}\vu\cdot \bfpsi_i \dx =& \int_{\mt} \vu_0\cdot \bfpsi_i \dx- \int_0^\cdot\varphi_R(\| \vu \|_{2,\infty})\int_{\mt}  \Big[\vu\cdot \Grad \vu + r[\vu] \Grad r[\vu ] \Big]\cdot \bfpsi_i   \dx\dt
\\
&+\int_0^\cdot\varphi_R(\| \vu \|_{2,\infty})\int_{\mt} D(r[\vu]) \Div \mathbb{S} (\Grad \vu)\cdot \bfpsi_i \dx\dt\label{est3}\\  &
+  \int_0^\cdot\varphi_R(\| \vu \|_{2,\infty}) \int_{\mt}  D(r[\vu]) \Div \mathbb{T}\cdot \bfpsi_i \dx\dt+\int_0^\cdot\int_{\mt}\bff\cdot\bfpsi\dx\dt ,\quad i = 1, \dots, n.\nonumber
\end{align}
is a contraction on $\mathcal B=C([0,T^\ast]; X_n)$ for $T^\ast$ sufficiently small. 
For $r_1 = r[\vc{v}_1]$, $r_2 = r[\vc{v}_2]$, we get
\begin{equation*} \label{T8}
\begin{split}
\D (r_1 - r_2) &+ \vc{v}_1 \cdot \Grad (r_1 - r_2) \dt - \frac{\gamma - 1}{2} \Div \vc{v}_1 (r_1 - r_2) \dt \\
&= - \Grad r_2 \cdot (\vc{v}_1 - \vc{v}_2) - \frac{\gamma - 1}{2}  r_2 \Div (\vc{v}_1 - \vc{v}_2) \dt,
\end{split}
\end{equation*}
where we have set
\[
\vc{v}_1 = \varphi_R(\| \vu_1 \|_{2,\infty}) \vu_1 , \ \vc{v}_2 = \varphi_R(\| \vu_2 \|_{2,\infty})\vu_2.
\]
Consequently, we easily deduce that
\begin{align}\label{eq:new}
\sup_{0\leq t\leq T^\ast}\big\|r[\vu_1 ]-r[\vu_2 ]\big\|^2_{L^2}\leq T^\ast c(n,R,T)\sup_{0\leq t\leq T^\ast} \big\| \vu_1 - \vu_2 \big\|_{X_n}^2
\end{align}
noting that $r_1$, $r_2$ coincide at $t=0$ and that $r_j$, $\Grad r_j$ are bounded by a constant depending on $R$, recall \eqref{est1}.
As a consequence of \eqref{est1}, \eqref{eq:new} and the equivalence of norms on $X_n$ we can show that the mapping
$\mathscr T$ satisfies the estimate
\begin{align}\label{Tdet}
\|\mathscr{T}\vu_1 - \mathscr{T}\vu_2 \|_\mathcal{B}^2\leq T^{\ast}c(n,R,T)\| \vu_1 - \vu_2 \|_{\mathcal B}^2.
\end{align}
The inequality \eqref{Tdet} shows that
$\mathscr T$ is a contraction provided we choose $T^\ast> 0$ small enough. 
Due to the cut-offs all nonlinearities in \eqref{est2} are globally Lipschitz continuous. Taking also into account that all norms on the finite-dimensional space $X_n$ are equivalent it is easy to show that
\begin{align*}
\sup_{0\leq t\leq T^*}\|\bfu^n\|_{X_n}^2\leq\, c(n,R,T,\bff,\mathbb T,\bfu^n_0)
\end{align*}
using $\bfu^n$ as  test-function.
Consequently, a continuity argument can be used to extend the solution $(r^n,\bu^n)$ to the whole interval $[0,T]$.

\subsection{A priori estimates}
\label{sec:exisNS}
Consider the solution $(r^n,\bu^n)$ to \eqref{E3'}, \eqref{est2} constructed in the previous section.
From the maximum principle, we gain the estimate 
\begin{align}
\label{maxPrinciple}
 r^n(t,\bx)\leq \sup_{\bx \in \tor} r_0(\bx)\exp\bigg(\int_0^t\varphi_R (\|\vu^n \|_{2,\infty})\|\nabx \bu^n \|_{L^\infty_{\mathbf{x}}}\dt\bigg)\leq \sup_{\bx \in \tor} r_0(\bx) \exp(cR\,t) \leq K_R
\end{align}
as well as
\begin{align}
\label{minPrinciple}
 r^n(t,\bx)\geq \inf_{\bx \in \tor} r_0(\bx)\exp\bigg(-\int_0^t\varphi_R (\|\vu^n \|_{2,\infty})\|\nabx \bu^n \|_{L^\infty_{\mathbf{x}}}\dt\bigg)\geq \inf_{\bx \in \tor} r_0(\bx) \exp(-cR\,t) \geq \frac{1}{K_R}
\end{align}
hold for all $(t,\bx)\in [0,T] \times \mt$ with a constant $K_R$. 
Similarly, we have
\begin{align}
\label{maxPrinciple2}
\begin{aligned}
 |\nabx r^n(t,\bx)|&\leq \sup_{\bx \in \tor} |\nabx r_0(\bx)|\exp\bigg(\int_0^t\varphi_R (\|\vu^n \|_{2,\infty})\| \bu^n \|_{W^{2,\infty}_{\mathbf{x}}}\dt\bigg)\\&\leq \sup_{\bx \in \mt} r_0(\bx) \exp(cR\,t) \leq K_R.
 \end{aligned}
\end{align}
Furthermore, per the definition \eqref{transfRhoToR} and the pressure law \eqref{isentropicPressure}, it follows from \eqref{maxPrinciple}--\eqref{maxPrinciple2} that
\begin{align}
\label{lipschitzDR}
\Vert D(r^n) \Vert_{W^{1,\infty}_\bx} + \Vert D(r^n)^{-1} \Vert_{L^{\infty}_\bx} \leq \,cK_R.
\end{align}
Additionally, it follows from \eqref{commutatorEstimateContinuous} and \eqref{maxPrinciple} that
\begin{align}
\label{dOfr}
\Vert D(r^n) \Vert_{W^{s,2}_\bx} \leq\,cK_R\,\Vert r^n\Vert_{W^{s,2}_\bx}.
\end{align}
With the above preparation, we now derive uniform a priori bounds for the solution $(r^n, \bu^n)$ of the coupled mass and momentum balance equations \eqref{E3'}, \eqref{est2} given $(r_0, \mathbf{u}_0, \mathbb T, \mathbf{f})$ as assumed in Theorem \ref{thm:main}. In particular, we suppose
$s > \frac{d}{2}+2$. In what follows, the constants hidden in $\lesssim$  depend on $R$
only through the constant $K_R$ from \eqref{maxPrinciple}--\eqref{dOfr} but are otherwise independent of $R$.
We proceed by applying $\partial^\alpha_\mathbf{x}$ to \eqref{E3'}, \eqref{est2}. This yield
\begin{equation}
\begin{aligned}
\label{contEqMultiIndex}
\partial_t \partial^\alpha_\mathbf{x} r^n  + \varphi_R (\|\vu^n \|_{2,\infty})\bigg[\mathbf{u}^n   \cdot \nabx  \partial^\alpha_\mathbf{x} r  ^n &+ \frac{\gamma-1}{2} r^n\, \divx   \partial^\alpha_\mathbf{x} \bu^n\bigg]
= J_1^n  + J_2^n 
\end{aligned}
\end{equation}
in $\mathcal Q_T$, where
\begin{align*}
J_1^n &= \varphi_R (\|\vu^n \|_{2,\infty})\bigg[\bu^n \cdot \partial^\alpha_\mathbf{x} \nabx    r^n - \partial^\alpha_\mathbf{x}(\bu^n \cdot \nabx   r^n)\bigg],\\
J_2^n &= \varphi_R (\|\vu^n \|_{2,\infty})\bigg[\frac{\gamma-1}{2}
\big( 
r^n   \partial^\alpha_\mathbf{x} \divx   \bu^n  - \partial^\alpha_\mathbf{x}(r^n \divx   \bu^n )  \big)\bigg],
\end{align*}
and
\begin{equation}
\begin{aligned}
\label{momEquMultiIndex}
\partial_t \partial^\alpha_\mathbf{x} \bu^n  + \varphi_R (\|\vu^n \|_{2,\infty})\big[\bu^n  \cdot \nabx  \partial^\alpha_\mathbf{x} \bu^n   &+ r^n\nabx   \partial^\alpha_\mathbf{x}r^n - D(r^n)\divx   \mathbb{S}(\nabx   \partial^\alpha_\mathbf{x} \bu^n )  
- 
D(r^n )\divx   \partial^\alpha_\mathbf{x} \mathbb{T}\big]
\\
&
=
 \partial^\alpha_\mathbf{x} \mathbf{f} + I_1^n  + I_2^n  + I_3^n  + I_4^n 
\end{aligned}
\end{equation}
in $X_n'$, where
\begin{align*}
&I_1^n  = \varphi_R (\|\vu^n \|_{2,\infty})\big[\bu^n  \cdot \partial^\alpha_\mathbf{x} \nabx   \bu^n   - \partial^\alpha_\mathbf{x}(\bu^n  \cdot \nabx   \bu^n  ) \big],
\\
&I_2^n  = \varphi_R (\|\vu^n \|_{2,\infty})\big[r^n\partial^\alpha_\mathbf{x} \nabx   r^n
-
 \partial^\alpha_\mathbf{x}(r^n   \nabx   r^n )\big],
 \\
&I_3^n  =\varphi_R (\|\vu^n \|_{2,\infty})\big[ - D(r^n ) \partial^\alpha_\mathbf{x}  \divx   \mathbb{S}(\nabx   \bu^n  ) + \partial^\alpha_\mathbf{x} \big( D(r^n )  \divx   \mathbb{S}(\nabx   \bu^n  ) \big)\big],
 \\
&I_4^n  = \varphi_R (\|\vu^n \|_{2,\infty})\big[- D(r^n ) \partial^\alpha_\mathbf{x}  \divx   \mathbb{T} + \partial^\alpha_\mathbf{x} \big( D(r^n )  \divx   \mathbb{T} \big) \big],
\end{align*}
respectively. 
We obtain from \eqref{commutatorEstimate} the estimates
\begin{equation}
\begin{aligned}
\label{j1}
\Vert J_1^n \Vert_{L^2_\bx} 
&\lesssim \varphi_R (\|\vu^n \|_{2,\infty})\Big(
\big\Vert \nabx  \bu^n   \big\Vert_{L^\infty_\bx}
\big\Vert \nabx  ^s  r^n   \big\Vert_{L^2_\bx}
+
\big\Vert \nabx   r^n  \big\Vert_{L^\infty_\bx}
\big\Vert \nabx  ^s \bu^n  \big\Vert_{L^2_\bx}\Big),
\end{aligned}
\end{equation}
 and
\begin{equation}
\begin{aligned}
\label{j2}
\Vert J_2^n  \Vert_{L^2_\bx} 
&\lesssim \varphi_R (\|\vu^n \|_{2,\infty})\Big(
\big\Vert \nabx   r^n  \big\Vert_{L^\infty_\bx}
\big\Vert \nabx  ^s  \bu^n   \big\Vert_{L^2_\bx}
+
\big\Vert \divx   \bu^n   \big\Vert_{L^\infty_\bx}
\big\Vert \nabx  ^s  r^n  \big\Vert_{L^2_\bx}\Big),
\end{aligned}
\end{equation}
for the right-hand side of \eqref{contEqMultiIndex}. We also obtain the estimates
\begin{equation}
\begin{aligned}
\label{i123}
\Vert I_1^n  \Vert_{L^2_\bx} 
&\lesssim \varphi_R (\|\vu^n \|_{2,\infty})
\big\Vert \nabx  \bu^n   \big\Vert_{L^\infty_\bx}
\big\Vert \nabx  ^{s}\bu^n   \big\Vert_{L^2_\bx},
\\
\Vert I_2^n  \Vert_{L^2_\bx} 
&\lesssim \varphi_R (\|\vu^n \|_{2,\infty})
\big\Vert \nabx   r^n  \big\Vert_{L^\infty_\bx}
\big\Vert \nabx  ^{s} r^n  \big\Vert_{L^2_\bx},
\\
\Vert I_3^n  \Vert_{L^2_\bx} 
&\lesssim\varphi_R (\|\vu^n \|_{2,\infty})\Big(
\big\Vert \nabx   D(r^n ) \big\Vert_{L^\infty_\bx}
\big\Vert  \nabx  ^{s} \mathbb{S}(\nabx   \bu^n  ) \big\Vert_{L^2_\bx}
+
\big\Vert \divx   \mathbb{S}(\nabx   \bu^n  )  \big\Vert_{L^\infty_\bx}
\big\Vert \nabx  ^s D(r^n ) \big\Vert_{L^2_\bx} \Big),
\end{aligned}
\end{equation}
for the indicated terms on the right-hand side of \eqref{momEquMultiIndex}. For the last term in \eqref{momEquMultiIndex} we have
\begin{equation}
\begin{aligned}
\label{i4}
\Vert I_4^n  \Vert_{L^2_\bx} 
&\leq \varphi_R (\|\vu^n \|_{2,\infty})\Big(
\big\Vert \nabx   D(r^n) \big\Vert_{L^\infty_\bx}
\big\Vert  \nabx  ^{s} \mathbb{T} \big\Vert_{L^2_\bx}
+
\big\Vert \divx   \mathbb{T}  \big\Vert_{L^\infty_\bx}
\big\Vert \nabx  ^s D(r^n) \big\Vert_{L^2_\bx}\Big).
\end{aligned}
\end{equation}
 Now if we multiply \eqref{contEqMultiIndex} by $ \partial^\alpha_\mathbf{x} r^n$  and integrate by parts when necessary, we gain
\begin{equation}
\begin{aligned}
\label{conEqMultiIndex1}
\frac{\mathrm{d}}{\mathrm{d}t}\int_{\mt}& \big\vert \partial^\alpha_\mathbf{x} r^n  \big\vert^2 \, \mathrm{d}\mathbf{x}
+ (\gamma -1)\varphi_R (\|\vu^n \|_{2,\infty})
\int_{\mt}   r^n  \divx   \partial^\alpha_\mathbf{x} \bu^n   \,\partial^\alpha_\mathbf{x} r^n  \, \mathrm{d}\mathbf{x}
\\&
= \varphi_R (\|\vu^n \|_{2,\infty})\int_{\mt} \big\vert \partial^\alpha_\mathbf{x} r^n   \big\vert^2 \divx   \bu^n   \, \mathrm{d}\mathbf{x}
+2
\int_{\mt} ( J_1^n  + J_2^n ) \,  \partial^\alpha_\mathbf{x} r^n  \, \mathrm{d}\mathbf{x}.
\end{aligned}
\end{equation}
It follows from \eqref{j1}, \eqref{j2} and \eqref{maxPrinciple2} that,
\begin{equation}
\begin{aligned}
\label{conEqMultiIndex2}
 \varphi_R (\|\vu^n \|_{2,\infty})&\int_{\mt} \big\vert \partial^\alpha_\mathbf{x} r^n   \big\vert^2 \divx   \bu^n   \, \mathrm{d}\mathbf{x}
+2
\int_{\mt} ( J_1^n  + J_2^n ) \,  \partial^\alpha_\mathbf{x} r^n \, \mathrm{d}\mathbf{x}\\
&= 
\int_{\mt} \partial^\alpha_\mathbf{x} r^n   \big(\varphi_R (\|\vu^n \|_{2,\infty}) \partial^\alpha_\mathbf{x} r^n  \divx   \bu^n   
+2
 ( J_1^n  + J_2^n )\big) \, \mathrm{d}\mathbf{x} 
\\
&\lesssim\varphi_R (\|\vu^n \|_{2,\infty})
 \big\Vert \partial^\alpha_\mathbf{x} r^n  \big\Vert_{L^2_\mathbf{x}} \Big( \Vert \bu^n   \Vert_{W^{1,\infty}_\mathbf{x}} \Vert r^n  \Vert_{W^{s,2}_\mathbf{x}}
+\Vert r^n  \Vert_{W^{1,\infty}_\mathbf{x}} \Vert \bu^n   \Vert_{W^{s,2}_\mathbf{x}} \Big)\\
&\lesssim
 R\big\Vert(r^n,\bu^n)  \big\Vert_{W^{s,2}_\mathbf{x}}^2
\end{aligned}
\end{equation}
with a constant depending only on $s,d$, $T$ and $K_R$. 
By substituting \eqref{conEqMultiIndex2} into \eqref{conEqMultiIndex1}, integrating in time and summing over $\alpha$ such that $\vert\alpha\vert\leq s$, we obtain for any $t>0$,
\begin{equation}
\begin{aligned}
\label{conEqMultiIndex3}
 \Vert  r^n (t) \Vert^2_{W^{s,2}_\mathbf{x}} 
&+ (\gamma -1)\varphi_R (\|\vu^n \|_{2,\infty}) \sum_{\vert \alpha\vert \leq s}
\inttO   r^n  \,\divx   \partial^\alpha_\mathbf{x} \bu^n   \,\partial^\alpha_\mathbf{x} r^n  \, \mathrm{d}\mathbf{x} \, \mathrm{d}\sigma\\
&\lesssim
\Vert  r_0 \Vert^2_{W^{s,2}_\mathbf{x}} 
+
R\int_0^t \Vert (r^n , \bu^n  ) \Vert^2_{W^{s,2}_\mathbf{x}}   \, \mathrm{d}\sigma.
\end{aligned}
\end{equation}
Now if we test \eqref{momEquMultiIndex} by $\partial^\alpha_\mathbf{x} \bu^n\in X_n$, integrate by parts where necessary and then integrate in time, we gain
\begin{equation}
\begin{aligned}
\label{momEqMultiIndex3}
   \big\Vert \partial^\alpha_\mathbf{x} \bu^n  (t) \big\Vert^2_{L^2_\mathbf{x}}
&-2\varphi_R (\|\vu^n \|_{2,\infty})
\inttO  r^n  \, \divx   \partial^\alpha_\mathbf{x} \bu^n   \, \partial^\alpha_\mathbf{x} r^n \, \mathrm{d}\mathbf{x}\, \mathrm{d}\sigma\\
&+2\varphi_R (\|\vu^n \|_{2,\infty})
\inttO  D(r^n ) \,
\mathbb{S}( \nabx   \partial^\alpha_\mathbf{x} \bu^n  )  : \nabx  \partial^\alpha_\mathbf{x} \bu^n    \, \mathrm{d}\mathbf{x} \, \mathrm{d}\sigma
\\&
=
 \big\Vert \partial^\alpha_\mathbf{x} \bu_0^n   \big\Vert^2_{L^2_\mathbf{x}}
 +
\varphi_R (\|\vu^n \|_{2,\infty})\inttO   \big\vert \partial^\alpha_\mathbf{x} \bu^n   \big\vert^2 \divx   \bu^n   \, \mathrm{d}\mathbf{x}\, \mathrm{d}\sigma\\
&-2\varphi_R (\|\vu^n \|_{2,\infty})\inttO  \nabx   D(r^n )  \partial^\alpha_\mathbf{x} \mathbb{T} \cdot \partial^\alpha_\mathbf{x} \bu^n    \, \mathrm{d}\mathbf{x} \, \mathrm{d}\sigma
\\&-2
\varphi_R (\|\vu^n \|_{2,\infty})\inttO  D(r^n ) \, \partial^\alpha_\mathbf{x} \mathbb{T}
 : \nabx  \partial^\alpha_\mathbf{x} \bu^n    \, \mathrm{d}\mathbf{x} \, \mathrm{d}\sigma\\
&-
2\varphi_R (\|\vu^n \|_{2,\infty})\inttO  \nabx   D(r^n ) \mathbb{S}( \nabx   \partial^\alpha_\mathbf{x} \bu^n  ) \cdot \partial^\alpha_\mathbf{x} \bu^n    \, \mathrm{d}\mathbf{x} \, \mathrm{d}\sigma
\\&+2\varphi_R (\|\vu^n \|_{2,\infty})
\inttO  \nabx   r^n  \cdot \partial^\alpha_\mathbf{x} r^n  \partial^\alpha_\mathbf{x} \bu^n   \, \mathrm{d}\mathbf{x} \, \mathrm{d}\sigma\\
&+
2\inttO (   \partial^\alpha_\mathbf{x}\mathbf{f} + I_1^n  + I_2^n  + I_3^n +I_4^n ) \cdot  \partial^\alpha_\mathbf{x} \bu^n   \, \mathrm{d}\mathbf{x} \, \mathrm{d}\sigma
\\&=: K_1^n  + \ldots + K_7^n.
\end{aligned}
\end{equation}
We can estimate the $K_i^n$s  as follows. Firstly, we have by \eqref{lipschitzDR}
\begin{equation}
\begin{aligned}
\label{k2est}
\vert K_2^n \vert 
 \lesssim
R\int_0^t  \big\Vert \partial^\alpha_\mathbf{x} \bu^n   \big\Vert^2_{L^2_\mathbf{x}}
 \, \mathrm{d}\sigma.
\end{aligned}
\end{equation}
Furthermore, we can estimate $K_3^n$ using \eqref{maxPrinciple} by
\begin{equation}
\begin{aligned}
\vert K_3^n \vert 
&\lesssim \int_0^t\Vert   \mathbb T\Vert_{W^{s,2}_\bx}\Vert  r^n\Vert_{W^{1,\infty}_\bx}\Vert   \vu^n \Vert_{W^{s,2}_\bx} \,\mathrm{d}\sigma \lesssim
 \int_0^t
\Vert   \mathbb T\Vert_{W^{s,2}_\bx}^2 \, \mathrm{d}\sigma 
+
\int_0^t \Vert  (r^n,\bu^n)   \Vert_{W^{s,2}_\mathbf{x}}^2\,\mathrm{d}\sigma.
\end{aligned}
\end{equation}
Also, by \eqref{maxPrinciple2} we have that
\begin{equation}
\begin{aligned}
\vert K_4^n \vert &\leq\,c K_R
\inttO |  \partial^\alpha_\mathbf{x} \mathbb{T}|  |\nabx  \partial^\alpha_\mathbf{x} \bu^n|   \, \mathrm{d}\mathbf{x} \, \mathrm{d}\sigma 
\leq
c_\delta K_R^2
 \int_0^t
\Vert   \mathbb T\Vert_{W^{s,2}_\mathbf{x}}^2 \, \mathrm{d}\sigma  
+ \delta  \int_0^t
\Vert   \bu^n\Vert_{W^{s+1,2}_\mathbf{x}}^2 \, \mathrm{d}\sigma  
\end{aligned}
\end{equation}
where $\delta>0$ is arbitrary.
By using \eqref{lipschitzDR}, 
\begin{equation}
\begin{aligned}
\label{k4est}
\vert K_5^n \vert &\leq\,cK_R
\inttO  | \mathbb{S}( \nabx  \partial^\alpha_\mathbf{x} \bu^n) |  \, | \partial^\alpha_\mathbf{x} \bu^n|   \, \mathrm{d}\mathbf{x} \, \mathrm{d}\sigma 
\leq \delta  \int_0^t
\Vert   \bu^n\Vert_{W^{s+1,2}_\mathbf{x}}^2 \, \mathrm{d}\sigma  
+ c_\delta K_R^2
\int_0^t   \Vert  \bu^n   \Vert_{W^{s,2}_\mathbf{x}}^2   \, \mathrm{d}\sigma
\end{aligned}
\end{equation}
whereas by \eqref{maxPrinciple2}
\begin{equation}
\begin{aligned}
\vert K_6^n \vert
&\lesssim
\int_0^t 
 \Big( \big\Vert \partial^\alpha_\mathbf{x} r^n  \big\Vert_{L^2_\mathbf{x}}^2
 +
\Vert  \partial^\alpha_\mathbf{x} \bu^n   \Vert^2_{L^{2}_\mathbf{x}} \Big)
\Vert \nabx   r^n   \Vert_{L^\infty_\mathbf{x}}
 \, \mathrm{d}\sigma 
 \lesssim \int_0^t \Vert (r^n , \bu^n  ) \Vert^2_{W^{s,2}_\mathbf{x}}   \, \mathrm{d}\sigma.
\end{aligned}
\end{equation}
Finally, we use \eqref{i123}--\eqref{i4} to obtain the following estimate
\begin{equation}
\begin{aligned}
\vert K_7^n \vert &\lesssim  
\int_0^t \varphi_R (\|\vu^n \|_{2,\infty})\Vert  \bu^n   \Vert_{W^{s,2}_\mathbf{x}} \Big(
\Vert  \mathbf{f} \Vert_{W^{s,2}_\mathbf{x}}
+
\Vert  I_1^n  \Vert_{L^{2}_\mathbf{x}} 
+
\Vert I_2^n  \Vert_{L^{2}_\mathbf{x}} 
+
\Vert  I_3^n  \Vert_{L^{2}_\mathbf{x}} 
+
\Vert  I_4^n  \Vert_{L^{2}_\mathbf{x}}  \Big)  \, \mathrm{d}\sigma
\\&
 \lesssim  
\int_0^t \varphi_R (\|\vu^n \|_{2,\infty})\Vert  \bu^n   \Vert_{W^{s,2}_\mathbf{x}} \Big(
\Vert  \mathbf{f} \Vert_{W^{s,2}_\mathbf{x}}
+
\Vert \bu^n   \Vert_{W^{1,\infty}_\mathbf{x}}
\Vert \bu^n   \Vert_{W^{s,2}_\mathbf{x}}
+
\Vert r^n  \Vert_{W^{1,\infty}_\mathbf{x}}
\Vert  r^n  \Vert_{W^{s,2}_\mathbf{x}}
\\&
+
\big\Vert \nabx   D(r^n ) \big\Vert_{L^\infty_\mathbf{x}}
\big\Vert   \mathbb{S}(\nabx   \bu^n  ) \big\Vert_{W^{s,2}_\mathbf{x}}
+
\big\Vert \divx   \mathbb{S}(\nabx   \bu^n  )  \big\Vert_{L^\infty_\mathbf{x}}
\big\Vert  D(r^n ) \big\Vert_{W^{s,2}_\mathbf{x}}   
\\&
+
\big\Vert \nabx   D(r^n ) \big\Vert_{L^\infty_\mathbf{x}}
\big\Vert   \mathbb{T} \big\Vert_{W^{s,2}_\mathbf{x}}
+
\big\Vert \divx   \mathbb{T}  \big\Vert_{L^\infty_\mathbf{x}}
\big\Vert  D(r^n ) \big\Vert_{W^{s,2}_\mathbf{x}}  \Big)
\, \mathrm{d}\sigma.
\end{aligned}
\end{equation}
It therefore follows from  \eqref{maxPrinciple}--\eqref{dOfr} that for any $\delta>0$, we can find $c_\delta=c_\delta(s,d,K_R)>0$ such that
\begin{equation}
\begin{aligned}
\label{k6est}
\vert K_7^n \vert 
 \leq  c_\delta
\int_0^t 
\Vert  \mathbf{f} \Vert_{W^{s,2}_\mathbf{x}}^2
\, \mathrm{d}\sigma+c_\delta R^2
\int_0^t  
\Vert (r^n, \bu^n ) \Vert^2_{W^{s,2}_\mathbf{x}}  
\, \mathrm{d}\sigma
+c_\delta
\int_0^t  
\Vert   \mathbb T\Vert_{W^{s,2}_\mathbf{x}}^2
\, \mathrm{d}\sigma
+
\delta \int_0^t  
\Vert   \bu^n\Vert_{W^{s+1,2}_\mathbf{x}}^2
\, \mathrm{d}\sigma.
\end{aligned}
\end{equation}
If we now substitute \eqref{k2est}--\eqref{k6est} into \eqref{momEqMultiIndex3} with a small enough choice of $\delta$, sum over $\alpha$ such that $\vert\alpha\vert\leq s$ and use again \eqref{maxPrinciple}, we obtain
\begin{equation}
\begin{aligned}
\label{momEqMultiIndex4}
  \Vert \bu^n (t) \Vert^2_{W^{s,2}_\mathbf{x}}
&-
2\varphi_R (\|\vu^n \|_{2,\infty})\sum_{\vert \alpha\vert \leq s}
\inttO  r^n \, \divx   \partial^\alpha_\mathbf{x} \bu^n  \, \partial^\alpha_\mathbf{x} r^n\, \mathrm{d}\mathbf{x}\, \mathrm{d}\sigma
+
\int_0^t \Vert \bu^n \Vert^2_{W^{s+1,2}_\mathbf{x}}\mathrm{d}\sigma
\\&
\lesssim  
\Vert \bu^n_0  \Vert^2_{W^{s,2}_\mathbf{x}}
+
\int_0^t  \Big(
\Vert  \mathbf{f} \Vert_{W^{s,2}_\mathbf{x}}^2
+
\Vert   \mathbb T\Vert_{W^{s,2}_\mathbf{x}}^2
\Big)
 \, \mathrm{d}\sigma +R^2\int_0^t
\Vert (r^n, \bu^n ) \Vert^2_{W^{s,2}_\mathbf{x}}   \, \mathrm{d}\sigma.
\end{aligned}
\end{equation}
Now if we  multiply \eqref{momEqMultiIndex4}  by $\frac{\gamma -1}{2}$ (which is always positive) and  sum the resulting inequality with \eqref{conEqMultiIndex3}, we obtain
\begin{equation}
\begin{aligned}
\label{sumEst}
   &\big\Vert \big(r^n(t), \bu^n (t)\big) \big\Vert^2_{W^{s,2}_\mathbf{x}}
+
\int_0^t \Vert \bu^n \Vert^2_{W^{s+1,2}_\mathbf{x}}\mathrm{d}\sigma
\\&\quad\lesssim
 \big\Vert \big(r^n_0, \bu^n_0  \big) \big\Vert^2_{W^{s,2}_\mathbf{x}}+\int_0^t  \Big(
\Vert  \mathbf{f} \Vert_{W^{s,2}_\mathbf{x}}^2
+
\Vert   \mathbb T\Vert_{W^{s,2}_\mathbf{x}}^2
\Big)
 \, \mathrm{d}\sigma +R^2\int_0^t
\Vert (r^n, \bu^n ) \Vert^2_{W^{s,2}_\mathbf{x}}   \, \mathrm{d}\sigma
 \end{aligned}
\end{equation}
with a constant $c$ depending only on $\gamma,\delta, s,d$ and $K_R$. Now since 
\begin{align*}
g_1(t):= c \bigg[
 \big\Vert \big(r_0, \mathbf{u}_0  \big) \big\Vert^2_{W^{s,2}_\mathbf{x}}
+ 
\int_0^t  \Big(
\Vert  \mathbf{f} \Vert_{W^{s,2}_\mathbf{x}}^2
+
\Vert   \mathbb T\Vert_{W^{s,2}_\mathbf{x}}^2\Big)
 \, \mathrm{d}\sigma
 \bigg]
\end{align*}
is non-decreasing in $t$,  it follows from Gronwall's lemma that
\begin{equation}
\begin{aligned}
\label{gronwellRandV}
&\sup_{t\in (0,T)}  \big\Vert (r^n, \bu^n  ) \big\Vert^2_{W^{s,2}_\mathbf{x}}
+ \int_0^T\Vert    \bu^n\Vert_{W^{s+1,2}_\mathbf{x}}^2\dt 
\leq g_1(T) \exp(cR^2T)
\end{aligned}
\end{equation}
for any given $T>0$. Recall that the constant $c$ in $g_1$ only depends on $R$ via the constant $K_R$
from \eqref{maxPrinciple}--\eqref{dOfr}.
It is now standard to pass to the limit in \eqref{E3'}, \eqref{est2} in order to obtain a global-in-time solution $(r_R,\bfu_R)$ to \eqref{E3'}, \eqref{est2} with
\begin{align*}
r_R \in  C([0,T]; W^{s,2}(\tor)), \ r>0,\quad
 \vu_R \in C([0,T]; W^{s,2}(\tor;\mathbb{R}^d))\cap L^2(0,T;W^{s+1,2}(\mt;\mathbb{R}^d)).
\end{align*}
For any fixed $R> \Vert\bu_0   \Vert_{W^{s,2}_\mathbf{x}} $ we can find $T_0\ll1$ such that
\begin{align}
\label{lipschitzRV0}
\Vert\bu_R(t)   \Vert_{W^{s,2}_\mathbf{x}} 
\leq R
\end{align}
for all $t\leq T_0$. Consequently, $(r_R,\bfu_R)$ is a solution to 
\eqref{E3}--\eqref{E4} as the cut-offs are not seen. Moreover, we can assume (by further decreasing $T_0$ if necessary) that the constant $K_R$ in the limit version of \eqref{maxPrinciple}--\eqref{dOfr} can be replaced by a uniform constant that is 
independent of $R$. This implies the required a priori estimate
 and finishes the proof of Theorem \ref{thm:main}. Note that the constant in the energy estimate \eqref{est_thm:main} does not depend on the initial datum. However, the maximal existence time $T_0$ does.

\subsection{Difference estimate for fluid system}
\label{sec:uniqNS}
The purpose of this subsection is to show uniqueness of solutions to \eqref{E3}--\eqref{E4}
and hence
prove Theorem \ref{thm:main'}. Based on the estimates from the previous subsection (that is Theorem \ref{thm:main}) the constructed solutions possess enough regularity provided the existence interval is chosen small enough.
Let $(r^i,\bu^i)$, $i=1,2$ be two solutions of \eqref{E3}--\eqref{E4} with data $(r_0,\vu_0,\bff,\mathbb T^1)$ and $(r_0,\vu_0,\bff,\mathbb T^2)$ respectively defined in an interval $[0,T_0]$.
We set
\begin{align}\label{RR}
\mathfrak R:=\max\bigg\{\sup_{0<t<T_0}\|\bu^1\|_{s,2},\sup_{0<t<T_0}\|\bu^2\|_{s,2}\bigg\}.
\end{align}
We obtain versions of \eqref{maxPrinciple}--\eqref{dOfr} for $r^1$ and $r^2$ with a constant $K_{\mathfrak R}$. In particular, we have for $r=r^1,r^2$
\begin{align}
\label{maxPrincipleR}
 &\frac{1}{K_{\mathfrak R}}\leq \inf_{\bx \in \tor} r_0(\bx) \exp(-c\mathfrak R\,t) \leq r(t,\bx)\leq \sup_{\bx \in \tor} r_0(\bx) \exp(c\mathfrak R\,t) \leq K_{\mathfrak R},\\
\label{maxPrinciple2R}
 &\qquad\qquad|\nabx r(t,\bx)|\leq \sup_{\bx \in \mt} r_0(\bx) \exp(c\mathfrak R\,t) \leq K_{\mathfrak R},\\
\label{lipschitzDRR}
&\Vert D(r) \Vert_{W^{1,\infty}_\bx} + \Vert D(r)^{-1} \Vert_{L^{\infty}_\bx} \leq \,cK_{\mathfrak R},\quad
\Vert D(r) \Vert_{W^{s,2}_\bx} \leq\,cK_{\mathfrak R}\,\Vert r\Vert_{W^{s,2}_\bx}.
\end{align}
In the following we derive estimates for the difference of $r^1$ and $r^2$. As in Section \ref{sec:exisNS}, the constants  only depend on $\mathfrak R$ via $K_{\mathfrak R}$ but are otherwise independent of $\mathfrak R$.\\
Set $r^{12}=r^1 -r^2$, $\bu^{12}=\bu^1 -\bu^2$,  $\mathbb{T}^{12}=\mathbb{T}^1 -\mathbb{T}^2$ so that  $(r^{12}, \bu^{12})$ satisfies
\begin{align}
&\partial_t r^{12} + \mathbf{u}^{12}\cdot \nabx   r^1 + \mathbf{u}^2\cdot \nabx   r^{12} +\frac{\gamma -1}{2} r^1\, \divx   \mathbf{u}^{12} +\frac{\gamma -1}{2} r^{12}\, \divx   \mathbf{u}^2 =0, \label{contEquRnewZ}
\\
&\partial_t \mathbf{u}^{12}  + \mathbf{u}^1 \cdot  \nabx  \mathbf{u}^{12} + \mathbf{u}^{12} \cdot  \nabx  \mathbf{u}^2 + r^1\nabx   r^{12} + r^{12}\nabx   r^2= D(r^1)\divx   \mathbb{S}(\nabx   \mathbf{u}^{12}) 
\nonumber \\
&\quad+ (D(r^1)-  D(r^2))\divx   \mathbb{S}(\nabx   \mathbf{u}^2) +D(r^1)\divx   \mathbb{T}^{12} +(D(r^1)-D(r^2))\divx   \mathbb{T}^2, \label{momEquRnewZ} 
\end{align}
subject to the following initial  condition
\begin{align}
&(r^{12}, \mathbf{u}^{12})\vert_{t=0}=(r^{12}_0, \mathbf{u}^{12}_0) =(0, 0)
&\quad \text{in }  \tor.
\end{align}
Let the multi-index $\alpha$ satisfy
\begin{align}
\label{alphaSprimeSminus}
\vert \alpha \vert \leq s' \leq s-1
\end{align}
with $s>\frac{d}{2}+2$.
By applying $\partial^\alpha_\mathbf{x}$ to  \eqref{contEquRnewZ}, we obtain
\begin{equation}
\begin{aligned}
\label{contEqMultiIndexnew}
\partial_t\partial^\alpha_\bx r^{12} &+ \mathbf{u}^{12}\cdot \nabx   \partial^\alpha_\bx r^1 + \mathbf{u}^2\cdot \nabx  \partial^\alpha_\bx r^{12} +\frac{\gamma -1}{2} r^1\, \divx   \partial^\alpha_\bx \mathbf{u}^{12} +\frac{\gamma -1}{2} r^{12}\, \divx  \partial^\alpha_\bx \mathbf{u}^2 
\\&
= J_1^a + J_1^b + J_2^a + J_2^b
\end{aligned}
\end{equation}
such that
\begin{align}
\Vert J_1^a \Vert_{L^2_\bx} 
&\lesssim
\big\Vert \nabx  \mathbf{u}^{12}   \big\Vert_{L^\infty_\bx}
\big\Vert \nabx^{s'}  r^1   \big\Vert_{L^2_\bx}
+
\big\Vert \nabx   r^1  \big\Vert_{L^\infty_\bx}
\big\Vert \nabx^{s'} \mathbf{u}^{12}  \big\Vert_{L^2_\bx},
\label{j1anew}
\\
\Vert J_1^b \Vert_{L^2_\bx} 
&\lesssim
\big\Vert \nabx  \mathbf{u}^2   \big\Vert_{L^\infty_\bx}
\big\Vert \nabx^{s'}  r^{12}   \big\Vert_{L^2_\bx}
+
\big\Vert \nabx   r^{12}  \big\Vert_{L^\infty_\bx}
\big\Vert \nabx^{s'} \mathbf{u}^2  \big\Vert_{L^2_\bx},
\label{j1bnew}
\\
\Vert J_2^a \Vert_{L^2_\bx} 
&\lesssim
\big\Vert \nabx   r^1  \big\Vert_{L^\infty_\bx}
\big\Vert \nabx^{s'}  \mathbf{u}^{12}   \big\Vert_{L^2_\bx}
+
\big\Vert \divx   \mathbf{u}^{12}   \big\Vert_{L^\infty_\bx}
\big\Vert \nabx^{s'}  r^1  \big\Vert_{L^2_\bx},\label{j2anew}
\\
\Vert J_2^b \Vert_{L^2_\bx} 
&\lesssim
\big\Vert \nabx   r^{12}  \big\Vert_{L^\infty_\bx}
\big\Vert \nabx^{s'}  \mathbf{u}^2   \big\Vert_{L^2_\bx}
+
\big\Vert \divx   \mathbf{u}^2   \big\Vert_{L^\infty_\bx}
\big\Vert \nabx^{s'}  r^{12}  \big\Vert_{L^2_\bx}. \label{j2bnew}
\end{align}
On the other hand, the application of $\partial^\alpha_\mathbf{x}$ to  \eqref{momEquRnewZ} yields
\begin{equation}
\begin{aligned}
\label{momEquMultiIndexnew}
 &\partial_t \partial^\alpha_\bx \mathbf{u}^{12}  + \mathbf{u}^1 \cdot  \nabx  \partial^\alpha_\bx \mathbf{u}^{12} + \mathbf{u}^{12} \cdot  \nabx  \partial^\alpha_\bx \mathbf{u}^2 + r^1\nabx  \partial^\alpha_\bx r^{12} + r^{12}\nabx \partial^\alpha_\bx  r^2= D(r^1)\divx   \mathbb{S}(\nabx \partial^\alpha_\bx  \mathbf{u}^{12}) 
\nonumber \\
&+ (D(r^1)-  D(r^2))\divx   \mathbb{S}(\nabx   \partial^\alpha_\bx\mathbf{u}^2) +D(r^1)\divx  \partial^\alpha_\bx \mathbb{T}^{12} +(D(r^1)-D(r^2))\divx   \partial^\alpha_\bx \mathbb{T}^2
\\&+ I_1^a + I_1^b + I_2^a + I_2^b + I_3^a + I_3^b + I_4^a + I_4^b
\end{aligned}
\end{equation}
where
\begin{align}
\Vert I_1^a \Vert_{L^2_\bx} 
&\lesssim
\big\Vert \nabx  \mathbf{u}^1   \big\Vert_{L^\infty_\bx}
\big\Vert \nabx^{s'}\mathbf{u}^{12}   \big\Vert_{L^2_\bx} 
+
  \big\Vert \nabx  \mathbf{u}^{12}   \big\Vert_{L^\infty_\bx}
\big\Vert \nabx^{s'}\mathbf{u}^1   \big\Vert_{L^2_\bx},
\label{i1a}
\\
\Vert I_1^b \Vert_{L^2_\bx} 
&\lesssim
\big\Vert \nabx  \mathbf{u}^{12}   \big\Vert_{L^\infty_\bx}
\big\Vert \nabx^{s'}\mathbf{u}^2   \big\Vert_{L^2_\bx} 
+
  \big\Vert \nabx  \mathbf{u}^2   \big\Vert_{L^\infty_\bx}
\big\Vert \nabx^{s'} \mathbf{u}^{12} \big\Vert_{L^2_\bx},
\label{i1b}
\\
\Vert I_2^a \Vert_{L^2_\bx} 
&\lesssim
\big\Vert \nabx  r^1   \big\Vert_{L^\infty_\bx}
\big\Vert \nabx^{s'}  r^{12}   \big\Vert_{L^2_\bx} 
+
  \big\Vert \nabx  r^{12}   \big\Vert_{L^\infty_\bx}
\big\Vert \nabx^{s'}  r^1   \big\Vert_{L^2_\bx},
\label{i2a}
\\
\Vert I_2^b \Vert_{L^2_\bx} 
&\lesssim
\big\Vert \nabx  r^{12}   \big\Vert_{L^\infty_\bx}
\big\Vert \nabx^{s'}  r^2   \big\Vert_{L^2_\bx} 
+
  \big\Vert \nabx  r^2   \big\Vert_{L^\infty_\bx}
\big\Vert \nabx^{s'}  r^{12} \big\Vert_{L^2_\bx},
\label{i2b}
\end{align}
and
\begin{align}
\Vert I_3^a \Vert_{L^2_\bx} 
&\lesssim
\big\Vert \nabx   D(r^1 ) \big\Vert_{L^\infty_\bx}
\big\Vert  \nabx^{s'} \mathbb{S}(\nabx   \mathbf{u}^{12}  ) \big\Vert_{L^2_\bx}
+
\big\Vert \divx   \mathbb{S}(\nabx   \mathbf{u}^{12}  )  \big\Vert_{L^\infty_\bx}
\big\Vert \nabx^{s'} D(r^1 ) \big\Vert_{L^2_\bx},  
\label{i3a}\\
\Vert I_3^b \Vert_{L^2_\bx} 
&\lesssim
\big\Vert \nabx  ( D(r^1 ) -D(r^2 ))\big\Vert_{L^\infty_\bx}
\big\Vert  \nabx^{s'} \mathbb{S}(\nabx   \mathbf{u}^2  ) \big\Vert_{L^2_\bx} \nonumber
\\&+
\big\Vert \divx   \mathbb{S}(\nabx   \mathbf{u}^2  )  \big\Vert_{L^\infty_\bx}
\big\Vert \nabx^{s'} ( D(r^1 ) -D(r^2))\big\Vert_{L^2_\bx},  
\label{i3b}\\
\Vert I_4^a \Vert_{L^2_\bx} 
&\lesssim
\big\Vert \nabx   D(r^1) \big\Vert_{L^\infty_\bx}
\big\Vert  \nabx^{s'} \mathbb{T}^{12} \big\Vert_{L^2_\bx}
+
\big\Vert \divx   \mathbb{T}^{12}  \big\Vert_{L^\infty_\bx}
\big\Vert \nabx^{s'} D(r^1) \big\Vert_{L^2_\bx}, \label{i4a}
\\
\Vert I_4^b \Vert_{L^2_\bx} 
&\lesssim
\big\Vert \nabx    ( D(r^1 ) -D(r^2)) \big\Vert_{L^\infty_\bx}
\big\Vert  \nabx^{s'} \mathbb{T}^2 \big\Vert_{L^2_\bx}
+
\big\Vert \divx   \mathbb{T}^2  \big\Vert_{L^\infty_\bx}
\big\Vert \nabx^{s'}  ( D(r^1 ) -D(r^2))\big\Vert_{L^2_\bx} \label{i4b}.
\end{align}
Testing \eqref{contEqMultiIndexnew} with $\partial^\alpha_\bx r^{12}$ yields
\begin{equation}
\begin{aligned}
\label{conEqMultiIndex1new}
\frac{\mathrm{d}}{\mathrm{d}t}&\int_{\tor} \big\vert \partial^\alpha_\mathbf{x} r^{12}  \big\vert^2 \, \mathrm{d}\mathbf{x}
+ (\gamma -1)
\int_{\tor}     r^1  \,\divx   \partial^\alpha_\mathbf{x} \mathbf{u}^{12} \,\partial^\alpha_\mathbf{x} r^{12}  \, \mathrm{d}\mathbf{x}
= \int_{\tor} \big\vert \partial^\alpha_\mathbf{x} r^{12}   \big\vert^2 \divx   \mathbf{u}^2   \, \mathrm{d}\mathbf{x}
\\&
-
(\gamma -1)
\int_{\tor}   r^{12}  \,\divx   \partial^\alpha_\mathbf{x} \mathbf{u}^2  \,\partial^\alpha_\mathbf{x} r^{12}  \, \mathrm{d}\mathbf{x}
-2
\int_{\tor} \partial^\alpha_\mathbf{x} r^{12} \, \nabx \partial^\alpha_\mathbf{x} r^1   \cdot   \mathbf{u}^{12}   \, \mathrm{d}\mathbf{x}
\\&+2
\int_{\tor} ( J_1^a +J_1^b + J_2^a+ J_2^b) \cdot  \partial^\alpha_\mathbf{x} r^{12}  \, \mathrm{d}\mathbf{x}.
\end{aligned}
\end{equation}
Now note that from H\"older's inequality, \eqref{j1anew}, Young's inequality and Sobolev embedding $W^{s',2}_{\bx}\hookrightarrow W^{1,\infty}_{\bx}$, we can obtain the following estimate
\begin{equation}
\begin{aligned}
\label{j1aQt}
\int_{\Q_t}&J^a_1\cdot \partial^\alpha_\mathbf{x} r^{12} \dx \, \mathrm{d}\sigma
\leq
\int_0^t \big\Vert J^a_1\big\Vert_{L^2_\bx} \big\Vert \partial^\alpha_\mathbf{x} r^{12} \big\Vert_{L^2_\bx} \, \mathrm{d}\sigma
\lesssim
\int_0^t\Big( 
\big\Vert \nabx  \mathbf{u}^{12}   \big\Vert_{L^\infty_\bx}
\big\Vert \nabx^{s'}  r^1   \big\Vert_{L^2_\bx}
 \big\Vert \partial^\alpha_\mathbf{x} r^{12} \big\Vert_{L^2_\bx}
\\&
+
\big\Vert \nabx^{s'} \mathbf{u}^{12}  \big\Vert_{L^2_\bx}
\big\Vert \nabx   r^1  \big\Vert_{L^\infty_\bx}
 \big\Vert \partial^\alpha_\mathbf{x} r^{12} \big\Vert_{L^2_\bx}
\Big)\, \mathrm{d}\sigma
\lesssim
\int_0^t
\Vert (r^{12},\bu^{12})  \Vert_{W^{s',2}_\mathbf{x}}^2 
\Vert  r^1   \Vert_{W^{s',2}_\mathbf{x}}\,\mathrm{d}\sigma.
\end{aligned}
\end{equation}
By integrating \eqref{conEqMultiIndex1new} in time, we can treat the corresponding terms on the right side of the equation as  in \eqref{j1aQt}. Indeed, some are easier  to tackle.
Subsequently, it follows from \eqref{j1anew}--\eqref{j2bnew} and the continuous embedding $W^{s'+1,2}_\bx \hookrightarrow W^{s',2}_\bx$ that
\begin{equation}
\begin{aligned}
\label{conEqMultiIndex3new}
& \Vert  r^{12} (t) \Vert^2_{W^{s',2}_\mathbf{x}} 
+ (\gamma -1)\sum_{\vert \alpha\vert \leq s'}
\inttO    r^1  \,\divx   \partial^\alpha_\mathbf{x} \mathbf{u}^{12}  \,\partial^\alpha_\mathbf{x} r^{12}  \, \mathrm{d}\mathbf{x} \, \mathrm{d}\sigma
\\&\lesssim
\int_0^t \Vert (r^{12},\bu^{12})  \Vert_{W^{s',2}_\mathbf{x}}^2 
\Big( \Vert r^1  \Vert_{W^{s'+1,2}_\mathbf{x}}
+ \Vert \mathbf u^2 \Vert_{W^{s'+1,2}_\mathbf{x}} \Big) \, \mathrm{d}\sigma.
\end{aligned}
\end{equation} 
Note that we need $W^{s'+1,2}_\bx$-regularity of $r^1$ and $\bu^2$ in the above because of the second and third terms on the right-hand side of \eqref{conEqMultiIndex1new}. Also note that by \eqref{alphaSprimeSminus}, $s'+1\leq s$ and as such, $W^{s,2}_\bx$ is contained in $W^{s'+1,2}_\bx$.
Now if we also test \eqref{momEquMultiIndexnew} by $ \partial^\alpha_\mathbf{x} \mathbf{u}^{12}$ and integrate by parts where necessary, we gain
\begin{equation}
\begin{aligned}
\label{momEqMultiIndex2}
&\frac{\mathrm{d}}{\mathrm{d}t}\int_{\tor}   \big\vert \partial^\alpha_\mathbf{x} \mathbf{u}^{12}   \big\vert^2 \, \mathrm{d}\mathbf{x}
-
2
\int_{\tor}  r^1  \divx   \partial^\alpha_\mathbf{x} \mathbf{u}^{12}   \, \partial^\alpha_\mathbf{x} r^{12} \, \mathrm{d}\mathbf{x}
+2
\int_{\tor} D(r^1 ) \,
\mathbb{S}( \nabx   \partial^\alpha_\mathbf{x} \mathbf{u}^{12}  )  : \nabx   \partial^\alpha_\mathbf{x} \mathbf{u}^{12}   \, \mathrm{d}\mathbf{x}
\\&=
\int_{\tor}   \big\vert \partial^\alpha_\mathbf{x} \mathbf{u}^{12}   \big\vert^2 \divx   \mathbf{u}^1   \, \mathrm{d}\mathbf{x}
-2
\int_{\tor}   \bu^{12} \nabx   \partial^\alpha_\mathbf{x}\mathbf{u}^2 \cdot  \partial^\alpha_\mathbf{x} \mathbf{u}^{12}      \, \mathrm{d}\mathbf{x}
\\
&-2
\int_{\tor}   r^{12}    \nabx  \partial^\alpha_\mathbf{x} r^2 \cdot \partial^\alpha_\mathbf{x} \mathbf{u}^{12}   \, \mathrm{d}\mathbf{x}
-
2
\int_{\tor} D(r^1 )  \partial^\alpha_\mathbf{x} \mathbb{T}^{12} 
 : \nabx   \partial^\alpha_\mathbf{x} \mathbf{u}^{12}   \, \mathrm{d}\mathbf{x}
\\&-
2\int_{\tor}  \nabx   D(r^1 )\big[ \mathbb{S}( \nabx   \partial^\alpha_\mathbf{x} \mathbf{u}^{12}  ) + \partial^\alpha_\mathbf{x} \mathbb{T}^{12} \big] \cdot  \partial^\alpha_\mathbf{x} \mathbf{u}^{12}   \, \mathrm{d}\mathbf{x}
+2
\int_{\tor}  \nabx   r^1  \cdot \partial^\alpha_\mathbf{x} r^{12}  \partial^\alpha_\mathbf{x} \mathbf{u}^{12}   \, \mathrm{d}\mathbf{x}
\\&
-2
\int_{\tor} \big( D(r^1 ) -D(r^2) \big) \divx  
\mathbb{S}( \nabx   \partial^\alpha_\mathbf{x} \mathbf{u}^2  ) \cdot \partial^\alpha_\mathbf{x} \mathbf{u}^{12}   \, \mathrm{d}\mathbf{x}
\\&-2
\int_{\tor} \big( D(r^1 ) -D(r^2) \big) \divx\partial^\alpha_\mathbf{x} \mathbb{T}^2 
 \cdot \partial^\alpha_\mathbf{x} \mathbf{u}^{12}   \, \mathrm{d}\mathbf{x}
\\
&
+
2 \int_{\tor} (I_1^a + \ldots + I_4^b) \cdot  \partial^\alpha_\mathbf{x} \mathbf{u}^{12}   \, \mathrm{d}\mathbf{x}.
\end{aligned}
\end{equation}
We can now integrate \eqref{momEqMultiIndex2} in time and estimate the right-hand terms. First of all, we can use Sobolev's inequality to obtain,
\begin{align*}
\bigg\vert \int_{\Q_t}   \big\vert \partial^\alpha_\mathbf{x} \mathbf{u}^{12}   \big\vert^2 \divx   \mathbf{u}^1 \, \dx  \, \mathrm{d}\sigma \bigg\vert  
\lesssim
\int_0^t \Vert  \mathbf{u}^{12}  \Vert_{W^{s',2}_\mathbf{x}}^2\Vert \bu^1 \Vert_{W^{s',2}_\mathbf{x}}  \, \mathrm{d}\sigma
\end{align*}
whereas in combination with H\"older's inequality, 
\begin{equation}
\begin{aligned}
\label{estDifferently}
\bigg\vert
2
\int_{\Q_t}   \bu^{12} \nabx   \partial^\alpha_\mathbf{x}\mathbf{u}^2 \cdot  \partial^\alpha_\mathbf{x} \mathbf{u}^{12}      \, \mathrm{d}\mathbf{x} \, \mathrm{d}\sigma
\bigg\vert
\lesssim
\int_0^t \Vert  \mathbf{u}^{12}  \Vert_{W^{s',2}_\mathbf{x}}^2\Vert \bu^2 \Vert_{W^{s'+1,2}_\mathbf{x}}  \, \mathrm{d}\sigma.
\end{aligned}
\end{equation}
Similarly, by Young's inequality
\begin{equation}
\begin{aligned}
\bigg\vert
2
\int_{\Q_t}   r^{12}    \nabx  \partial^\alpha_\mathbf{x} r^2 \cdot \partial^\alpha_\mathbf{x} \mathbf{u}^{12}   \, \mathrm{d}\mathbf{x}  \, \mathrm{d}\sigma
\bigg\vert
\lesssim
\int_0^t \Vert  r^{12}  \Vert_{W^{s',2}_\mathbf{x}}^2\Vert r^2 \Vert_{W^{s'+1,2}_\mathbf{x}}  \, \mathrm{d}\sigma
+
\int_0^t \Vert  \mathbf{u}^{12}  \Vert_{W^{s',2}_\mathbf{x}}^2\Vert r^2 \Vert_{W^{s'+1,2}_\mathbf{x}}  \, \mathrm{d}\sigma.
\end{aligned}
\end{equation}
Using \eqref{maxPrincipleR}, Sobolev's embedding and \eqref{lipschitzDRR} yields
\begin{equation}
\begin{aligned}
\bigg\vert
2&
\int_{\Q_t} D(r^1 )  \partial^\alpha_\mathbf{x} \mathbb{T}^{12} 
 : \nabx   \partial^\alpha_\mathbf{x} \mathbf{u}^{12}   \, \mathrm{d}\mathbf{x}  \, \mathrm{d}\sigma
\bigg\vert
\leq 
\delta
\int_{0}^t\Vert  \bu^{12}  \Vert_{W^{s'+1,2}_\mathbf{x}}^2   \, \mathrm{d}\sigma+
c_\delta\int_0^t \Vert  \mathbb T^{12}  \Vert_{W^{s',2}_\mathbf{x}}^2 \, \mathrm{d}\sigma
\end{aligned}
\end{equation}
for any $\delta>0$.
Similarly, we obtain
\begin{equation}
\begin{aligned}
\bigg\vert
2\int_{\Q_t}  \nabx   D(r^1 ) \mathbb{S}( \nabx   \partial^\alpha_\mathbf{x} \mathbf{u}^{12}  )  &\cdot  \partial^\alpha_\mathbf{x} \mathbf{u}^{12}   \, \mathrm{d}\mathbf{x}  \, \mathrm{d}\sigma
\bigg\vert
\leq 
\delta
\int_{0}^t\Vert  \bu^{12}  \Vert_{W^{s'+1,2}_\mathbf{x}}^2   \, \mathrm{d}\sigma+
c_\delta\int_0^t \Vert  \bu^{12}  \Vert_{W^{s',2}_\mathbf{x}}^2 \, \mathrm{d}\sigma
\end{aligned}
\end{equation}
for any $\delta>0$. Also, \eqref{maxPrinciple2R} yields
\begin{equation}
\begin{aligned}
\bigg\vert
2&\int_{\Q_t}  \nabx   D(r^1 ) \,   \partial^\alpha_\mathbf{x} \mathbb{T}^{12}   \cdot  \partial^\alpha_\mathbf{x} \mathbf{u}^{12}   \, \mathrm{d}\mathbf{x}  \, \mathrm{d}\sigma
\bigg\vert
\lesssim
 \int_0^t
 \Vert  \mathbf{u}^{12}\Vert_{W^{s',2}_\mathbf{x}}^2   \, \mathrm{d}\sigma+
\int_0^t \Vert  \mathbb T^{12}  \Vert_{W^{s',2}_\mathbf{x}}^2\, \mathrm{d}\sigma
\end{aligned}
\end{equation}
and 
\begin{align}
\bigg\vert  2
\int_{\Q_t}  \nabx   r^1  \cdot \partial^\alpha_\mathbf{x} r^{12}  \partial^\alpha_\mathbf{x} \mathbf{u}^{12}   \, \mathrm{d}\mathbf{x}\, \mathrm{d}\sigma
\bigg\vert
\lesssim
\int_0^t \Vert  r^{12}  \Vert_{W^{s',2}_\mathbf{x}}^2  \, \mathrm{d}\sigma
+
\int_0^t \Vert  \mathbf{u}^{12}  \Vert_{W^{s',2}_\mathbf{x}}^2 \, \mathrm{d}\sigma.
\end{align}
As consequence of \eqref{lipschitzDRR} we obtain
\begin{equation}
\begin{aligned}
\bigg\vert
2
&\int_{\Q_t} \big( D(r^1 ) -D(r^2) \big) \divx  
\mathbb{S}( \nabx   \partial^\alpha_\mathbf{x} \mathbf{u}^2  )  \cdot \partial^\alpha_\mathbf{x} \mathbf{u}^{12}   \, \mathrm{d}\mathbf{x}\, \mathrm{d} \sigma\bigg\vert\\
&\lesssim
\int_0^t \Vert  \mathbf{u}^{12}  \Vert_{W^{s',2}_\mathbf{x}}\Vert  r^{12}  \Vert_{W^{s',2}_\mathbf{x}}\Vert \bu^2 \Vert_{W^{s'+2,2}_\mathbf{x}}  \, \mathrm{d}\sigma.
\end{aligned}
\end{equation}
In order to estimate the term involving $\mathbb T^2$ we integrate by parts to obtain
\begin{align*}
-2
\int_{\Q_t} \big( D(r^1 ) -D(r^2) \big) \divx\partial^\alpha_\mathbf{x} \mathbb{T}^2 
 \cdot \partial^\alpha_\mathbf{x} \mathbf{u}^{12}   \, \mathrm{d}\mathbf{x}\, \mathrm{d}\sigma
 &=
 2
\int_{\Q_t} \nabx\big( D(r^1 ) -D(r^2) \big)\cdot \partial^\alpha_\mathbf{x} \mathbb{T}^2 
 \partial^\alpha_\mathbf{x} \mathbf{u}^{12}   \, \mathrm{d}\mathbf{x}\, \mathrm{d}\sigma\\
 &+2
\int_{\Q_t} \big( D(r^1 ) -D(r^2) \big)\partial^\alpha_\mathbf{x} \mathbb{T}^2 
 : \partial^\alpha_\mathbf{x} \nabx\mathbf{u}^{12}   \, \mathrm{d}\mathbf{x}\, \mathrm{d}\sigma\\
 &=2
\int_{\Q_t} D'(r^1 )\big(\nabx r^1-\nabx r^2 \big)\cdot \partial^\alpha_\mathbf{x} \mathbb{T}^2 
 \partial^\alpha_\mathbf{x} \mathbf{u}^{12}   \, \mathrm{d}\mathbf{x}\, \mathrm{d}\sigma\\
  &+2
\int_{\Q_t} \big( D'(r^1 ) -D'(r^2) \big)\nabx r^2\cdot  \partial^\alpha_\mathbf{x} \mathbb{T}^2 
 \partial^\alpha_\mathbf{x} \mathbf{u}^{12}   \, \mathrm{d}\mathbf{x}\, \mathrm{d}\sigma\\
 &+2
\int_{\Q_t} \big( D(r^1 ) -D(r^2) \big)\partial^\alpha_\mathbf{x} \mathbb{T}^2 
 : \partial^\alpha_\mathbf{x} \nabx\mathbf{u}^{12}   \, \mathrm{d}\mathbf{x}\, \mathrm{d}\sigma
 \end{align*}
 Using \eqref{maxPrincipleR}, \eqref{maxPrinciple2R} and Sobolev's embedding we can estimate these terms as follows
 \begin{align*}
 \bigg\vert 2
\int_{\Q_t} D'(r^1 )\big(\nabx r^1-\nabx r^2 \big) \partial^\alpha_\mathbf{x} \mathbb{T}^2 
 \cdot \partial^\alpha_\mathbf{x} \mathbf{u}^{12}   \, \mathrm{d}\mathbf{x}\, \mathrm{d}\sigma \bigg\vert
 &\lesssim
 \int_0^t \Vert  \mathbf{u}^{12}  \Vert_{W^{s',2}_\mathbf{x}}\Vert  r^{12}  \Vert_{W^{s',2}_\mathbf{x}}\Vert \mathbb T^2 \Vert_{W^{s',2}_\mathbf{x}}  \, \mathrm{d}\sigma,\\
 \bigg\vert 2
\int_{\Q_t} \big( D'(r^1 ) -D'(r^2) \big)\nabx r^2 \partial^\alpha_\mathbf{x} \mathbb{T}^2 
 \partial^\alpha_\mathbf{x} \mathbf{u}^{12}   \, \mathrm{d}\mathbf{x}\, \mathrm{d}\sigma \bigg\vert
 &\lesssim
 \int_0^t \Vert  \mathbf{u}^{12}  \Vert_{W^{s',2}_\mathbf{x}}\Vert  r^{12}  \Vert_{W^{s',2}_\mathbf{x}}\Vert \mathbb T^2 \Vert_{W^{s',2}_\mathbf{x}}  \, \mathrm{d}\sigma,\\
 \bigg\vert2
\int_{\Q_t} \big( D(r^1 ) -D(r^2) \big)\partial^\alpha_\mathbf{x} \mathbb{T}^2 
 : \partial^\alpha_\mathbf{x} \nabx\mathbf{u}^{12}   \, \mathrm{d}\mathbf{x} \, \mathrm{d}\sigma \big\vert
 &\leq\,\delta\int_{0}^t\Vert  \bu^{12}  \Vert_{W^{s'+1,2}_\mathbf{x}}^2   \, \mathrm{d}\sigma+ c_\delta\int_0^t\Vert  r^{12}  \Vert_{W^{s',2}_\mathbf{x}}^2\Vert \mathbb T^2 \Vert_{W^{s',2}_\mathbf{x}}^2  \, \mathrm{d}\sigma.
 \end{align*}
As a result of \eqref{i1a}--\eqref{i2b} and the continuous embedding $W^{s'+1,2}_\bx\hookrightarrow W^{1,\infty}_\bx$, we also have
\begin{align*}
\bigg\vert
2 \int_{\Q_t} &(I_1^a + \ldots + I_2^b) \cdot  \partial^\alpha_\mathbf{x} \mathbf{u}^{12}   \, \mathrm{d}\mathbf{x}\, \mathrm{d}\sigma
\bigg\vert
\lesssim
\int_0^t
\Vert (r^{12} , \mathbf{u}^{12}  ) \Vert^2_{W^{s',2}_\mathbf{x}}   \Big( 
\Vert (r^1,r^2)  \Vert_{W^{s'+1,2}_\mathbf{x}}
+ \Vert (\mathbf{u}^1,\mathbf{u}^2)   \Vert_{W^{s'+1,2}_\mathbf{x}}\Big)
\, \mathrm{d}\sigma.
\end{align*}
Furthermore, due to \eqref{lipschitzDRR}
\begin{equation}
\begin{aligned}
\bigg\vert
2 \int_{\Q_t} I_3^a &\cdot  \partial^\alpha_\mathbf{x} \mathbf{u}^{12}   \, \mathrm{d}\mathbf{x}\, \mathrm{d}\sigma
\bigg\vert
\leq
\delta\int_0^t \big\Vert\mathbf{u}^{12}  \big\Vert_{W^{s'+1,2}_\bx}^2 \, \mathrm{d}\sigma+ c
\int_0^t
\Vert \mathbf{u}^{12}  \Vert^2_{W^{s',2}_\mathbf{x}}  
\Vert r^1  \Vert^2_{W^{s',2}_\mathbf{x}}
\, \mathrm{d}\sigma
\end{aligned}
\end{equation}
as well as
\begin{align}
\bigg\vert
2 \int_{\Q_t} I_3^b \cdot  \partial^\alpha_\mathbf{x} \mathbf{u}^{12}   \, \mathrm{d}\mathbf{x}\, \mathrm{d}\sigma\bigg\vert
\lesssim
\int_0^t \Vert  \mathbf{u}^{12}  \Vert_{W^{s',2}_\mathbf{x}}\Vert  r^{12}  \Vert_{W^{s',2}_\mathbf{x}}\Vert \bu^2 \Vert_{W^{s'+1,2}_\mathbf{x}}  \, \mathrm{d}\sigma
\end{align}
because of $W^{s'+1,2}_\bx\hookrightarrow W^{2,\infty}_\bx$.
Finally,
\begin{equation}
\begin{aligned}
\bigg\vert
2 \int_{\Q_t}  I_4^a \cdot  \partial^\alpha_\mathbf{x} \mathbf{u}^{12}   \, \mathrm{d}\mathbf{x}\, \mathrm{d}\sigma
\bigg\vert
&\lesssim
\int_0^t\Vert   \mathbb T^{12}\Vert_{W^{s',2}_\mathbf{x}}^2 \Vert r^1 \Vert_{W^{s',2}_\mathbf{x}}
\, \mathrm{d}\sigma
+
\int_0^t \Vert  \mathbf{u}^{12}  \Vert_{W^{s',2}_\mathbf{x}}^2 
\Vert r^1 \Vert_{W^{s',2}_\mathbf{x}}\, \mathrm{d}\sigma,
\\
\bigg\vert
2 \int_{\Q_t}  I_4^b \cdot  \partial^\alpha_\mathbf{x} \mathbf{u}^{12}   \, \mathrm{d}\mathbf{x}\, \mathrm{d}\sigma
\bigg\vert
&
\lesssim
\int_0^t
\Vert r^{12}  \Vert_{W^{s',2}_\mathbf{x}}\Vert  \mathbf{u}^{12} \Vert_{W^{s',2}_\mathbf{x}}
\Vert   \mathbb T^2\Vert_{W^{s',2}_\mathbf{x}}
\, \mathrm{d}\sigma.
\end{aligned}
\end{equation}
In conclusion
\begin{equation}
\begin{aligned}
\label{momEqMultiIndex4new}
&  \Vert \mathbf{u}^{12}  (t) \Vert^2_{W^{s',2}_\mathbf{x}}
-
2\sum_{\vert \alpha\vert \leq s'}
\inttO  r^1  \, \divx   \partial^\alpha_\mathbf{x} \mathbf{u}^{12}   \, \partial^\alpha_\mathbf{x} r^{12} \, \mathrm{d}\mathbf{x}\, \mathrm{d}\sigma
+
\int_0^t \Vert \mathbf{u}^{12}  (\sigma) \Vert^2_{W^{s'+1,2}_\mathbf{x}} \mathrm{d}\sigma
\\&
\lesssim
\int_0^t
\Vert (r^{12} , \mathbf{u}^{12}  ) \Vert^2_{W^{s',2}_\mathbf{x}}   \Big( 
\Vert (r^1,r^2)  \Vert^2_{W^{s'+1,2}_\mathbf{x}}
+ \Vert (\mathbf{u}^1,\mathbf{u}^2)   \Vert_{W^{s'+2,2}_\mathbf{x}}+
\Vert   \mathbb T^2\Vert_{W^{s',2}_\mathbf{x}}^2+1\Big)
\, \mathrm{d}\sigma
\\&+
\int_0^t\Vert   \mathbb T^{12}\Vert_{W^{s',2}_\mathbf{x}}^2 \Vert r^1 \Vert_{W^{s'+1,2}_\mathbf{x}}
\, \mathrm{d}\sigma.
\end{aligned}
\end{equation}
Summing \eqref{conEqMultiIndex3new} with the product of $\frac{\gamma-1}{2}$ and \eqref{momEqMultiIndex4new} then yield
\begin{equation}
\begin{aligned}
\label{momEqMultiIndex4anew}
&  \Vert ( r^{12}(t), \mathbf{u}^{12}(t))   \Vert^2_{W^{s',2}_\mathbf{x}}
+
\int_0^t \Vert \mathbf{u}^{12}  (\sigma) \Vert^2_{W^{s'+1,2}_\mathbf{x}}\, \mathrm{d}\sigma
\\&
\lesssim
\int_0^t
\Vert (r^{12} , \mathbf{u}^{12}  ) \Vert^2_{W^{s',2}_\mathbf{x}}   \Big( 
\Vert (r^1,r^2)  \Vert^2_{W^{s'+1,2}_\mathbf{x}}
+ \Vert (\mathbf{u}^1,\mathbf{u}^2)   \Vert_{W^{s'+2,2}_\mathbf{x}}+
\Vert   \mathbb T^2\Vert_{W^{s',2}_\mathbf{x}}^2+1\Big)
\, \mathrm{d}\sigma
\\&+
\int_0^t\Vert   \mathbb T^{12}\Vert_{W^{s',2}_\mathbf{x}}^2 \Vert r^1 \Vert_{W^{s'+1,2}_\mathbf{x}}
\, \mathrm{d}\sigma.
\end{aligned}
\end{equation}
Applying Gronwell's lemma finishes the proof of Theorem \ref{thm:main'} with a constant depending on $K_{\mathfrak{R}}$. This implies uniqueness. As in Section \ref{sec:exisNS} we can now decrease the time $T$ (depending on $\sup_{0<t<T_0}\|\bu^1\|_{s,2}$ and $\sup_{0<t<T_0}\|\bu^2\|_{s,2}$, recall the definition of $\mathfrak R$ in \eqref{RR}) such that the constant is independent of $\mathfrak R$ which completes the proof of Theorem \ref{thm:main'}.
\begin{remark}
We remind the reader that the condition \eqref{alphaSprimeSminus} means that in particular, $s'+2\leq s+1$ which is the threshold differentiability exponent of the solutions $\bu^i$, $i=1,2$ to the momentum equation \eqref{momEquR}. 
\end{remark} 

\section{Solving the Fokker--Planck equation}
\label{sec:FP}
\noindent
The aim of this section is to solve the Fokker--Planck equation
\begin{align}\label{eq:FP}
\partial_t \psi + \divx  (\mathbf{u} \psi) 
=
\varepsilon \Delx \psi
-
 \divq  \big( (\nabx   \mathbf{u}) \mathbf{q}\psi \big) 
 +
\frac{A_{11}}{4\lambda}  \divq  \bigg( M \nabq  \bigg(\frac{\psi}{M} \bigg)
\bigg)
\end{align} 
with $\varepsilon\geq0$ and $A_{11}, \lambda>0$. Here, $\bu$ is a given smooth function and we recall that the Maxwellian is given by
\begin{align*}
M(\mathbf{q}) = \frac{e^{-U \left(\frac{1}{2}\vert \mathbf{q} \vert^2 \right) }}{\int_Be^{-U \left(\frac{1}{2}\vert \mathbf{q} \vert^2 \right) }\,\mathrm{d}\mathbf{q}},\quad  U (s) = -\frac{b}{2} \log \bigg(1-  \frac{2s}{b} \bigg), \quad s\in [0,b/2)
\end{align*}
with $b>2$.
Let us start with a precise definition of the solution.
\begin{definition}
\label{def:strsolmartFP}
Let $s\in\mathbb N$ and $T>0$. Assume that
$\psi_0\in W^{s,2}\big( \tor; L^2_M(B) \big)$
and $\bfu\in C([0,T]; W^{s,2}(\tor ;\mathbb{R}^d )) \cap L^2(0,T; W^{s+1,2}(\tor ;\mathbb{R}^d ))$.
We call
$\psi$
a solution to the system \eqref{eq:FP} with initial condition $\psi_0$ in the interval $[0,T]$ provided the following holds.
\begin{itemize}
\item[(a)] $\psi$ satisfies
\begin{align*}
C \big( [0, T]; W^{s,2} \big( \tor; L^2_M(B) \big) \big)  
\cap L^2 \big( 0,T; W^{s,2}\big(\tor; H^1_M(B) \big)\big);
\end{align*}
\item[(b)] there holds for all $t\in[0,T]$  and for any $\phi(\bq) \in C^1(B)$,
\[
\begin{split}
\int_B\psi (t)\phi\dq  &= \int_B\psi_0\phi \dq - \int_0^t \int_B\Big[ \divx  (\mathbf{u} \psi) 
-
\varepsilon \,\Delx \psi
 \Big] \phi\dq\dif\sigma
+
\int_0^t \int_B
 (\nabx   \mathbf{u}) \mathbf{q}\psi \cdot \nabq \phi \, \dq\dif\sigma
\\&
-
\frac{A_{11}}{4\lambda}  \int_0^t\int_B
M \nabq  \bigg(\frac{\psi}{M} \bigg)\cdot \nabq \phi \,\dq\dif\sigma.
\end{split}
\]
\end{itemize}
\end{definition}
As in the case of Definition \ref{def:strsolmartRho}, differentiability in time as well as the correct initial datum follows from the equation in (b). 
We now formulate our results concerning well-posedness and uniqueness for \eqref{eq:FP}. 
\begin{theorem}\label{thm:mainFP}
Let $s\in\mathbb N$ satisfy $s>\frac{d}{2} + 2$, $T>0$ and suppose that $\varepsilon\geq0$. Assume that
$\psi_0\in W^{s,2}\big( \tor; L^2_M(B) \big)$
and $\bfu \in C([0,T_0]; W^{s,2}(\tor ;\mathbb{R}^d))\cap L^2(0,T_0; W^{s+1,2}(\tor ;\mathbb{R}^d))$. Then there is a solution $\psi$ to problem \eqref{eq:FP}
in the sense of Definition \ref{def:strsolmartFP} defined  on the interval $[0,T]$ and with the initial condition $\psi_0$. Moreover, we have the estimate
\begin{align}\label{eq:thm:mainFP}
\begin{aligned}
\sup_{0<t<T}\Vert\psi\Vert^2_{W^{s,2}_\mathbf{x}L^2_M}&+2\varepsilon\int_0^T\Vert\psi\Vert^2_{W^{s+1,2}_\mathbf{x}L^2_M}\,\mathrm{d}t
+
\bigg(\frac{A_{11}}{2\lambda} -4\delta \bigg) \int_0^T\Vert\psi\Vert^2_{W^{s,2}_\mathbf{x}H^1_M}\,\mathrm{d}t
\\
&\leq  c\exp\bigg(c\int_0^T\Vert   \bu  \Vert^2_{W^{s+1,2}_\bx}\,\mathrm{d}t \bigg) \Vert \psi_0\Vert^2_{W^{s,2}_\mathbf{x}L^2_M}.
\end{aligned}
\end{align}
with a constant depending only on $s,b,d,\delta$ and $T$.
\end{theorem}
The proof of Theorem \ref{thm:mainFP} can be found in Section \ref{sec:exisFokkerPlanck}.
\begin{theorem}\label{thm:main'FP}
Let the assumptions of Theorem \ref{thm:mainFP} be satisfied. Then the solution from Theorem \ref{thm:mainFP} is unique. Moreover, for $s'\in\mathbb N$ with $s' \leq s-1$ we have the estimate
\begin{equation}
\begin{aligned}
\label{est:thm:main'}
\sup_{0<t<T}&\Vert\psi^1-\psi^2\Vert^2_{W^{s',2}_\mathbf{x}L^2_M}
+
2\varepsilon\int_0^T\Vert\psi^1-\psi^2\Vert^2_{W^{s'+1,2}_\mathbf{x}L^2_M}
\mathrm{d}t
+
\bigg(\frac{A_{11}}{2\lambda} -8\delta \bigg)
\int_0^T \Vert\psi^1-\psi^2\Vert^2_{W^{s',2}_\mathbf{x}H^1_M} \mathrm{d}t
\\&
\leq\,c
\exp\bigg(c\int_0^T \Vert (\mathbf{u}^1,\mathbf{u}^2)   \Vert_{W^{s'+1,2}_\mathbf{x}}^2
\, \mathrm{d}t \bigg)\int_0^t\Vert \bu^{1}-\bu^2  \Vert^2_{W^{s'+1,2}_\bx}
\Vert(\psi^1,\psi^2)\Vert^2_{W^{s'+1,2}_\mathbf{x}L^2_M}\mathrm{d}t
\end{aligned}
\end{equation}
where $\psi^1$ and $\psi^2$ are two  solutions to \eqref{eq:FP}
with data $(\psi_0,\bu^1)$ and $(\psi_0,\bu^2)$ respectively.
\end{theorem}
The proof of Theorem \ref{thm:main'FP} can be found in Section \ref{sec:uniqFokkerPlanck}.

\subsection{A priori estimates}
\label{sec:exisFokkerPlanck}
In order to justify the following calculations, we need to work with an approximate system.
Following \cite{masmoudi2008well}, we consider an orthonormal basis
$\{\phi_n,\dots,\phi_n\}$ of eigenfunction of the operator $L \psi =- \divq  \big( M \nabq  \big(\frac{\psi}{M} \big)\big)$ on $L^2_M$ with the domain
\begin{align*}
D(L)=\bigg\{\psi \in L^2_M \, :\,  \psi\in H^1_M, \quad  \divq\bigg( M \nabq  \bigg(\frac{\psi}{M} \bigg) \bigg) \in L^2_M, \quad  M \nabq \bigg( \frac{\psi}{M} \bigg)\bigg\vert_{\partial B} =0
  \bigg\},
\end{align*} 
where the boundary condition is interpreted in the weak sense. We denote $Y_n = {\rm span} \left\{ \phi_1, \dots, \phi_n \right\}$.
Furthermore,
we introduce a cut-off in the last term in order to avoid the blow-up of $M$ for $\bq$ close to $\partial B$.
So let $\chi_n=\chi_n(\bq)\in C^1(B)$, $B:=B(\bm{0},\sqrt{b})$ where  $\chi_n=1$ in $B(\bm{0}, \sqrt{b}-\frac{2}{n})$ and  $\chi_n=0$ when $\vert \bq \vert \geq \sqrt{b}-\frac{1}{n}$.
It will be needed to ensure that certain terms belong to $L^2_M$.
Similar to Section \ref{subsec:galerkin}, we consider a smooth orthonormal basis $\left\{ \omega_m \right\}_{m=1}^\infty$ of the space $L^{2}(\mt)$ with $X_n = {\rm span} \left\{ \omega_1, \dots, \omega_n \right\}$. We denote by $\Pi^n_{\bx}:L^2(\mt)\rightarrow X_n$ and $\Pi^n_{\bq}:L^2_M(B)\rightarrow Y_n$
the corresponding orthogonal projections.
We aim to solve for $n\in\mathbb N$,
\begin{equation} \label{eq:FPn}
\begin{split}
\partial_t\int_{\mt \times B } \psi^n\, \phi_i \omega_j \dq &\dx
+  \int_{\mt \times B } \divx  \Big(M\mathbf{u} \chi_n\Pi^n_{\bq}\frac{\psi^n}{M}\Big)\, \phi_i \omega_j \dq \dx
\\
&=\int_{\mt \times B } \varepsilon \Delx \psi^n\, \phi_i\omega_j \dq\dx
-  \int_{\mt \times B }  \divq  \big( (\nabx   \mathbf{u}) \mathbf{q}\chi_n\psi^n \big)  \, \phi_i\omega_j \dq\dx
\\& +
\frac{A_{11}}{4\lambda} \int_{\mt \times B } \divq\bigg( M \nabq  \bigg(\frac{\psi^n}{M} \bigg) \bigg)\, \phi_i \omega_j \dq\dx,\quad i = 1, \dots, n,\ j=1,\dots,n.\\
\psi^n(0)&=\psi_0^n:=\Pi^n_{\bx}\Pi^n_{\bq} \psi_0.
\end{split}
\end{equation}
 Equation \eqref{eq:FPn} is an ODE which can be solved locally in time.
We start by showing
an estimate for $\psi^n$ in $L^2_\bx L^2_M$  which will imply global solvability of \eqref{eq:FPn}.
We can take $ \Pi^n_{\bq}\frac{\psi^n}{M}$ as test function in \eqref{eq:FPn} and integrate over $\Q_t \times B$ for which we obtain
\begin{equation}
\begin{aligned}
\label{distriFormPsiApprZ0}
&\frac{1}{2}\int_{\tor \times B} M  \bigg\vert\frac{ \psi^n(\sigma) }{M} \bigg\vert^2\dq \dx \bigg]_{\sigma=0}^{\sigma=t}     
+
\varepsilon
\int_{\Q_t \times B} M \bigg\vert \nabx \frac{\psi^n}{M} \bigg\vert^2 \dq  \dx \, \mathrm{d}\sigma
\\
&  
+
\frac{A_{11}}{4\lambda}
\int_{\Q_t \times B} M \bigg\vert \nabq \frac{\psi^n}{M} \bigg\vert^2 \dq  \dx \, \mathrm{d}\sigma
=
\int_{\Q_t \times B}\chi_nM (\nabx   \bu )\mathbf{q} \frac{\psi^n}{M} \cdot \nabq   \Pi^n_{\bq}\frac{\psi^n}{M} \dq \dx \, \mathrm{d}\sigma
\\&
+
\int_{\Q_t \times B}\frac{1}{2}\chi_n M (\divx \bu ) \bigg\vert  \Pi^n_{\bq}\frac{ \psi^n}{M} \bigg\vert^2 \dq \dx \, \mathrm{d}\sigma  
-
\int_{\Q_t \times B} M\chi_n \,\divx  \bu   \,  \Big|\Pi^n_{\bq}\frac{\psi^n}{M}\Big|^2 \dq \dx \, \mathrm{d}\sigma\\&= :(I)^n+(II)^n+(III)^n. 
\end{aligned}
\end{equation}
As a consequence of Young's inequality, H\"older's inequality and $\vert \bq\vert^2<b$, we obtain the estimates
\begin{align*}
(I)^n
&\leq
\delta \int_{\Q_t \times B} M
\bigg\vert
\nabq \frac{\psi^n}{M}  \bigg\vert^2 \dq  \dx \, \mathrm{d}\sigma
+
c_\delta(b)
\int_{\Q_t} \vert \nabx  \bu  \vert^2\int_{ B} M \bigg\vert  \frac{\psi^n}{M} \bigg\vert^2 \dq \dx\, \mathrm{d}\sigma,\\
(II)^n
&\leq 
\frac{1}{2}
\int_0^t \Vert \divx  \bu \Vert_{L^\infty_\bx}\int_{\tor \times B} M \bigg\vert \frac{\psi^n}{M} \bigg\vert^2 \dq \dx\, \mathrm{d}\sigma,\\
(III)^n
&\leq 
\frac{1}{2}
\int_{\Q_t \times B} M
\bigg\vert \frac{\psi^n}{M}  \bigg\vert^2\dq \dx\, \mathrm{d}\sigma
+
\frac{1}{2}
\int_{\Q_t}\vert \divx   \bu  \vert^2 \int_{ B} M \bigg\vert  \frac{\psi^n}{M} \bigg\vert^2 \dq \dx\, \mathrm{d}\sigma
\end{align*}
for any $\delta>0$. Note that we also took into account continuity of $\Pi_{\bq}^n$ on $L^2_M$ and $H_M^1$.
All three terms together can be bounded by
\begin{align*}
\delta \int_0^t\Vert\psi^n\Vert^2_{L^{2}_{\bx}H^1_M}\ds + c \int_0^t (1+\Vert \bu \Vert_{W^{s,2}_\bx}^2) \Vert \psi^n\Vert^2_{L^2_{\bx}L^2_M}\ds
\end{align*}
using Sobolev's embedding.
Plugging this into \eqref{distriFormPsiApprZ0} and using Gronwall's lemma yields
\begin{align}\label{suptx3Z}
\begin{aligned}
\sup_{0<t<T}\Vert \psi^n \Vert^2_{L^{2}_\mathbf{x}L^2_M}
&\leq  c\exp\bigg(c\int_0^T\Vert   \bu  \Vert^2_{W^{s,2}_\bx}\,\mathrm{d}\sigma\bigg)\Vert\psi_0\Vert^2_{L^{2}_\mathbf{x}L^2_M}
\end{aligned}
\end{align}
since $\varepsilon\geq0$ and $\frac{A_{11}}{4\lambda}-\delta>0$ for $\delta>0$ small. Note that \eqref{suptx3Z} implies that there is a global solution to \eqref{eq:FPn}.\\
Next apply $\partial^\alpha_\bx$ to \eqref{eq:FPn} to obtain
\begin{equation}
\begin{aligned}
\label{diffQuoZ}
\partial_t &\partial^\alpha_\bx\psi^n 
+ M\chi_n\Pi_{\bq}^n\frac{\psi^n}{M}\, \divx  \partial^\alpha_\bx \bu 
+ M\chi_n\bu \cdot \nabx \partial^\alpha_\bx \Pi_{\bq}^n\frac{\psi^n}{M} \\
&=
\varepsilon \Delx \partial^\alpha_\bx\psi^n
-
 \divq  \big( \chi_n(\nabx \partial^\alpha_\bx   \mathbf{u}) \mathbf{q} \, \psi^n \big) 
 + \frac{A_{11}}{4\lambda}
  \divq  \bigg(M \nabq \partial^\alpha_\bx \bigg(\frac{\psi^n}{M} \bigg)
\bigg)
+ K_1^n + K_2^n + \divq K_3^n
\end{aligned}
\end{equation}
in $(X_n\otimes Y_n)'$, where
\begin{align*}
&K_1^n = M\chi_n\Pi_{\bq}^n\frac{\psi^n}{M}\, \partial^\alpha_\bx \divx \bu - \partial^\alpha_\bx \Big[M\chi_n\Pi_{\bq}^n\frac{\psi^n}{M} \, \divx \bu\Big],
\\ 
&K_2^n =\bu \cdot \partial^\alpha_\bx  \nabx M\chi_n\Pi_{\bq}^n\frac{\psi^n}{M} - \partial^\alpha_\bx \Big[\bu \cdot \nabx M\chi_n \Pi_{\bq}^n\frac{\psi^n}{M}\Big],
\\
&K_3 = (\partial^\alpha_\bx  \nabx \bu) \bq\, \chi_n\psi^n - \partial^\alpha_\bx  [(\nabx \bu) \bq \, \chi_n\psi^n].
\end{align*}
As a consequence of \eqref{commutatorEstimate} and $\vert \bq\vert<\sqrt{b}$, it follows that
\begin{align}
&\Vert K_1^n \Vert_{L^2_\bx} \lesssim_{s,b}
M\Big\Vert \nabx \Pi_{\bq}^n\frac{\psi^n}{M} \Big\Vert_{L^\infty_\bx}
\Vert \nabx^s \bu \Vert_{L^2_\bx}  
+
M\Vert \divx \bu \Vert_{L^\infty_\bx}
\Big\Vert \nabx^s\Pi_{\bq}^n\frac{ \psi^n}{M} \Big\Vert_{L^2_\bx},
\label{commutatorK1Z}
\\
&\Vert K_2^n \Vert_{L^2_\bx} \lesssim_{s,b}
M\Vert \nabx \bu \Vert_{L^\infty_\bx}
\Big\Vert \nabx^s \Pi_{\bq}^n\frac{\psi^n}{M} \Big\Vert_{L^2_\bx}  
+
M\Big\Vert \nabx \Pi_{\bq}^n\frac{\psi^n}{M} \Big\Vert_{L^\infty_\bx}
\Vert \nabx^s \bu \Vert_{L^2_\bx},
\label{commutatorK2Z}
\\
&\Vert K_3^n \Vert_{L^2_\bx} \lesssim_{s,b}
\Vert \nabx \psi^n \Vert_{L^\infty_\bx}
\Vert \nabx^s \bu \Vert_{L^2_\bx}  
+
\Vert \nabx \bu \Vert_{L^\infty_\bx}
\Vert  \nabx^s \psi^n \Vert_{L^2_\bx}. \label{commutatorK3Z}
\end{align}
We multiply \eqref{diffQuoZ} by $\partial^\alpha_\bx \Pi^n_{\bq}\frac{\psi^n}{M}$ and integrate over $ \Q_t \times B$ to gain
\begin{equation}
\begin{aligned}
\label{distriFormPsiApprZ}
&\frac{1}{2}\int_{\tor \times B} M  \bigg\vert \partial^\alpha_\bx\frac{ \psi^n(\sigma) }{M} \bigg\vert^2\dq \dx \bigg]_{\sigma=0}^{\sigma=t}     
+
\varepsilon
\int_{\Q_t \times B} M \bigg\vert \nabx\partial^\alpha_\bx \frac{\psi^n}{M} \bigg\vert^2 \dq  \dx \, \mathrm{d}\sigma
\\
&  
+
\frac{A_{11}}{4\lambda}
\int_{\Q_t \times B} M \bigg\vert \nabq \partial^\alpha_\bx \frac{\psi^n}{M} \bigg\vert^2 \dq  \dx \, \mathrm{d}\sigma
=
\int_{\Q_t \times B}M (\nabx \partial^\alpha_\bx   \bu )\mathbf{q} \frac{\chi_n\psi^n}{M} \cdot \nabq  \partial^\alpha_\bx \Pi_{\bq}^n\frac{\psi^n}{M} \dq \dx \, \mathrm{d}\sigma
\\&
+
\int_{\Q_t \times B}\frac{1}{2} \chi_nM (\divx \bu ) \bigg\vert  \partial^\alpha_\bx\Pi_{\bq}^n\frac{ \psi^n}{M} \bigg\vert^2 \dq \dx \, \mathrm{d}\sigma  
-
\int_{\Q_t \times B}  M\chi_n \Pi^n_{\bq}\frac{\psi^n}{M}\,\divx  \partial^\alpha_\bx \bu   \cdot\partial^\alpha_\bx \Pi_{\bq}^n\frac{\psi^n}{M} \dq \dx \, \mathrm{d}\sigma 
\\
&+
\int_{\Q_t \times B}
(K_1^n+K_2^n) \cdot \partial^\alpha_\bx \Pi_{\bq}^n\frac{\psi^n}{M} 
\dq \dx \, \mathrm{d}\sigma
-
\int_{\Q_t \times B}
K_3^n\cdot \nabq  \partial^\alpha_\bx\Pi_{\bq}^n \frac{\psi^n}{M} 
\dq \dx \, \mathrm{d}\sigma.
\end{aligned}
\end{equation}
In order to estimate the terms on the right-hand side, we repeatedly  use the embedding $W^{s,2}_\bx \hookrightarrow W^{1,\infty}_\bx$ which follows from $s \geq 1+\frac{d}{2}$. Furthermore,
we use the continuity of $\Pi_{\bq}^n$ on $L^2_M$ and $H_M^1$.
By Young's and H\"older's inequalities (note that $\vert \bq\vert^2<b$), we obtain the estimate
\begin{equation}
\begin{aligned}
\label{Residual1aZ}
\bigg\vert &\int_{\Q_t \times B}M (\nabx \partial^\alpha_\bx   \bu )\mathbf{q} \frac{\chi_n \psi^n}{M} \cdot \nabq  \partial^\alpha_\bx \Pi_{\bq}^n \frac{\psi^n}{M} \dq  \dx \, \mathrm{d}\sigma \bigg\vert\\
&\leq
\delta \int_{\Q_t \times B} M
\bigg\vert
\nabq \partial^\alpha_\bx \frac{\psi^n}{M}  \bigg\vert^2 \dq  \dx \, \mathrm{d}\sigma
+
c_\delta
\int_{\Q_t} \vert \nabx \partial^\alpha_\bx  \bu  \vert^2\int_{ B} M \bigg\vert  \frac{\psi^n}{M} \bigg\vert^2 \dq \dx\, \mathrm{d}\sigma\\
&\leq \delta \int_0^t\Vert\psi^n\Vert^2_{W^{s,2}_{\bx}H^1_M}\ds + c_\delta\int_0^t \Vert\bu\Vert_{W^{s+1,2}_\bx}\Vert\psi^n\Vert^2_{W^{s,2}_{\bx}L^2_M}\ds
\end{aligned}
\end{equation}
for any $\delta>0$ as well as
\begin{equation}
\begin{aligned}
\label{higherStep3aZ}
\bigg\vert \int_{\Q_t \times B}  \frac{1}{2}\chi_n M (\divx \bu ) \bigg\vert  \partial^\alpha_\bx \Pi_{\bq}^n\frac{ \psi^n}{M} \bigg\vert^2  \dq \dx\, \mathrm{d}\sigma \bigg\vert
&\leq 
\frac{1}{2}
\int_0^t \Vert \divx  \bu \Vert_{L^\infty_\bx}\int_{\tor \times B} M \bigg\vert \partial^\alpha_\bx \frac{\psi^n}{M} \bigg\vert^2 \dq \dx\, \mathrm{d}\sigma\\&\lesssim \int_0^t \Vert\bu\Vert_{W^{s,2}_\bx}\Vert\psi^n\Vert^2_{W^{s,2}_{\bx}L^2_M}\ds.
\end{aligned}
\end{equation}
Furthermore, we have
\begin{equation}\label{eq:1010}
\begin{aligned}
\bigg\vert &
\int_{\Q_t \times B} M \chi_n \Pi_{\bq}^n \frac{\psi^n}{M} \,\divx  \partial^\alpha_\bx \bu   \cdot \partial^\alpha_\bx \Pi_{\bq}^n\frac{\psi^n}{M} \dq \dx\, \mathrm{d}\sigma  
\bigg\vert\\
&\leq 
\frac{1}{2}
\int_{\Q_t \times B} M
\bigg\vert
 \partial^\alpha_\bx \frac{\psi^n}{M}  \bigg\vert^2\dq \dx\, \mathrm{d}\sigma
+
\frac{1}{2}
\int_{\Q_t}\vert \divx \partial^\alpha_\bx  \bu  \vert^2 \int_{ B} M \bigg\vert  \frac{\psi^n}{M} \bigg\vert^2 \dq \dx\, \mathrm{d}\sigma\\
&\lesssim\int_0^t\Vert\psi^n\Vert^2_{W^{s,2}_{\bx}L^2_M}\ds+\int_0^t \Vert\bu\Vert_{W^{s+1,2}_\bx}\Vert\psi^n\Vert^2_{W^{s,2}_{\bx}L^2_M}\ds
\end{aligned}
\end{equation}
and  for any $\delta>0$,
\begin{equation}
\begin{aligned}
\label{Residual12aZ}
\bigg\vert 
\int_{\Q_t \times B}&
K^n_3\cdot \nabq  \partial^\alpha_\bx\Pi_{\bq}^n \frac{\psi^n}{M} 
\dq \dx\, \mathrm{d}\sigma
 \bigg\vert \\
&\leq 
c_\delta \int_{\Q_t \times B}
\frac{1}{M}\vert K^n_3\vert^2 \
\dq \dx\, \mathrm{d}\sigma
+\delta
\int_{\Q_t \times B}
M \bigg\vert \nabq  \partial^\alpha_\bx \frac{\psi^n}{M} \bigg\vert^2
\dq \dx\, \mathrm{d}\sigma\\
&= 
c_\delta \int_{\Q_t \times B}
\frac{1}{M}\vert K^n_3\vert^2 \
\dq \dx\, \mathrm{d}\sigma+ \delta \int_0^t\Vert\psi^n\Vert^2_{W^{s,2}_{\bx} H^1_M}\ds.
\end{aligned}
\end{equation}
Analogously, we use Young's inequality to estimate $K^n_1$ and $K^n_2$.
Substituting \eqref{Residual1aZ}--\eqref{Residual12aZ} into \eqref{distriFormPsiApprZ} and choosing $\delta$ small enough yields
\begin{equation}
\begin{aligned}
\label{suptxZ} 
\Vert\psi^n (\sigma)\Vert^2_{W^{s,2}_{\bx}L^2_M}\Big]_{\sigma=0}^{\sigma=t}
&+2\varepsilon 
\int_0^t\Vert\psi^n\Vert^2_{W^{s+1,2}_{\bx}L^2_M}\, \mathrm{d}\sigma
+ \bigg(\frac{A_{11}}{2\lambda} - 4\delta \bigg) \int_0^t\Vert\psi^n\Vert^2_{W^{s,2}_{\bx}H^1_M}\ds
\\&
\lesssim \int_0^t\Vert\psi^n\Vert^2_{W^{s,2}_{\bx}L^2_M}\ds+\int_0^t \Vert\bu\Vert^2_{W^{s+1,2}_\bx}\Vert\psi^n\Vert^2_{W^{s,2}_{\bx}L^2_M}\ds\\
&+
\int_{\Q_t \times B}
\frac{1}{M}\big( \vert K^n_1\vert^2 +\vert K^n_2\vert^2+\vert K^n_3\vert^2 \big) 
\dq \dx\, \mathrm{d}\sigma.
\end{aligned}
\end{equation}
On the other hand, by \eqref{commutatorK1Z}--\eqref{commutatorK3Z} we have
\begin{equation}
\begin{aligned}
\label{suptxKZ}
\sum_{i=1}^3
\int_{\Q_t \times B}
\frac{1}{M} \vert K^n_i\vert^2 
\dq \dx \,\mathrm{d}\sigma
&\lesssim
\int_0^t \Vert  \nabx^s  \bu  \Vert^2_{L^2_\bx} \bigg(\sup_{\bx \in \tor}\int_{ B}M\bigg\vert \nabx \frac{  \psi^n  }{M} \bigg\vert^2 \dq\bigg)\,\mathrm{d}\sigma 
\\&
+ 
\int_0^t  \Vert \nabx  \bu \Vert_{L^\infty_\bx}^2
\int_{\tor\times B}M\bigg\vert \nabx^s \frac{  \psi^n  }{M} \bigg\vert^2 \dq \dx\,\mathrm{d}\sigma
\\ &\lesssim
\int_0^t \Vert  \bu  \Vert^2_{W^{s,2}_\bx}
\Vert\psi^n\Vert^2_{W^{s,2}_{\bx}L^2_M}\ds.
\end{aligned}
\end{equation}
Substituting \eqref{suptxKZ} into \eqref{suptxZ} yields
\begin{equation}
\begin{aligned}
\label{suptx4Z}
\Vert\psi^n(\sigma)\Vert^2_{W^{s,2}_{\bx}L^2_M}\Big]_{\sigma=0}^{\sigma=t}
&+2\varepsilon 
\int_0^t\Vert\psi^n\Vert^2_{W^{s+1,2}_{\bx}L^2_M}\ds
+ \bigg(\frac{A_{11}}{2\lambda} - 4\delta \bigg)
\int_0^t\Vert\psi^n\Vert^2_{W^{s,2}_{\bx}H^1_M}\ds
\\&
\lesssim\int_0^t (1+\Vert \bu \Vert^2_{W^{s+1,2}_\bx})\Vert\psi^n\Vert^2_{W^{s,2}_{\bx}L^2_M}\ds .
\end{aligned}
\end{equation}
By Gronwall's lemma we obtain
\begin{align}\label{eq:1010b}
\begin{aligned}
\sup_{0<t<T}\Vert\psi^n (t)\Vert^2_{W^{s,2}_\mathbf{x}L^2_M}&+2\varepsilon\int_0^T   \Vert \psi^n \Vert^2_{W^{s+1,2}_\mathbf{x}L^2_M}\,\mathrm{d}t
+ \bigg(\frac{A_{11}}{2\lambda} - 4\delta \bigg)
\int_0^T\Vert\psi^n\Vert^2_{W^{s,2}_\mathbf{x}H^1_M}\,\mathrm{d}t\\
&\leq  c \exp\bigg(c \int_0^T\Vert   \bu  \Vert^2_{W^{s+1,2}_\bx}\,\mathrm{d}t\bigg)\Vert\psi_0\Vert^2_{W^{s,2}_\mathbf{x}L^2_M}
\end{aligned}
\end{align}
uniformly in $n$. It is now standard to pass to the limit $n\rightarrow\infty$ in \eqref{eq:FPn} and to show that
\eqref{eq:1010b} also holds in the limit. The proof of Theorem \ref{thm:mainFP} is hereby complete.

\subsection{Difference estimate for the Fokker-Planck equation}
\label{sec:uniqFokkerPlanck}
Given the estimates from the previous section (see, in particular, Theorem \ref{thm:mainFP}) the solution possesses enough regularity to justify the following calculations.
Let $\psi^1$ and $\psi^2$ be two solutions of \eqref{eq:FP} with data $(\bu^1,\psi_0)$ and $(\bu^2,\psi_0)$ respectively.
We set $\psi^{12}=\psi^1 -\psi^2$ and $\bu^{12}=\bu^1 -\bu^2$ so that  $\psi^{12}$ solves
\begin{align}
\label{fokkerPlanknew0}
\partial_t \psi^{12} + \divx  (\mathbf{u}^2 \psi^{12}) + \divx  (\mathbf{u}^{12} \psi^1) 
=
\varepsilon \Delx \psi^{12}
-
 \divq  \big( (\nabx   \mathbf{u}^1) \mathbf{q}\psi^{12} \big) 
\nonumber\\
-
 \divq  \big( (\nabx   \mathbf{u}^{12}) \mathbf{q}\psi^2 \big) 
 +
\frac{A_{11}}{4\lambda}  \divq  \bigg( M \nabq  \bigg(\frac{\psi^{12}}{M} \bigg)
\bigg)
\end{align}
subject to the following initial and boundary conditions
\begin{align}
& \mathbf{u}^{12}\vert_{t=0}= \mathbf{u}^{12}_0 =0
&\quad \text{in }  \tor,
\label{initialDensityVelonew0}
\\
&\psi^{12} \vert_{t=0}=\psi^{12}_0 = 0
& \quad \text{in } \tor \times B,
\label{fokkerPlankIintialnew0}
\\
&\psi^{12} =0
&\quad \text{on }  [0,T] \times \tor \times \partial B. \label{fokkerPlankBoundarynew0}
\end{align}
We now wish to establish a priori estimates in the spirit of Section \ref{sec:exisFokkerPlanck} for the difference $\psi^{12}$. We will repeatedly use the embedding
$W^{s',2}_{\bx}\hookrightarrow W^{1,\infty}_{\bx}$. We start by proving a counterpart of \eqref{suptx3Z}. Testing \eqref{fokkerPlanknew0} with $ \frac{\psi^{12}}{M}$ and integrating over $\Q_t \times B$ yield
\begin{equation}
\begin{aligned}
\label{distriFormPsiApprZnew0}
&\frac{1}{2}\int_{\tor \times B} M  \bigg\vert \frac{ \psi^{12}(t)}{M} \bigg\vert^2 \dq \dx  
+
\varepsilon
\int_{\Q_t \times B} M \bigg\vert \nabx \frac{\psi^{12}}{M} \bigg\vert^2 \dq  \dx \, \mathrm{d}\sigma
\\&  
+
\frac{A_{11}}{4\lambda}
\int_{\Q_t \times B} M \bigg\vert \nabq  \frac{\psi^{12}}{M} \bigg\vert^2 \dq  \dx\, \mathrm{d}\sigma  
=
\int_{\Q_t \times B}M (\nabx \bu^1 )\mathbf{q} \frac{\psi^{12}}{M} \cdot \nabq \frac{\psi^{12}}{M} \dq \dx\, \mathrm{d}\sigma
\\&  
+
\int_{\Q_t \times B}M (\nabx \bu^{12} )\mathbf{q} \frac{\psi^2}{M} \cdot \nabq \frac{\psi^{12}}{M} \dq \dx\, \mathrm{d}\sigma
+
\int_{\Q_t \times B}\frac{1}{2} M (\divx \bu^2 ) \bigg\vert \frac{ \psi^{12}}{M} \bigg\vert^2 \dq \dx\, \mathrm{d}\sigma 
\\&
-
\int_{\Q_t \times B} M  \bu^{12}\cdot \nabx \frac{ \psi^1}{M}\cdot \frac{ \psi^{12}}{M}  \dq \dx \, \mathrm{d}\sigma  
-
\int_{\Q_t \times B}  \psi^{12} \,\divx \bu^2   \cdot \frac{\psi^{12}}{M} \dq \dx \, \mathrm{d}\sigma  
\\&
-
\int_{\Q_t \times B}  \psi^1 \,\divx  \bu^{12}   \cdot  \frac{\psi^{12}}{M} \dq \dx \, \mathrm{d}\sigma=:J_1+\dots J_6.
\end{aligned}
\end{equation}
The terms on the right-hand side can be estimates as follows:
\begin{align*}
J_1&\lesssim \int_0^t\|\nabla_x\bu^1\|_{L^\infty_{\bx}}\Vert\psi^{12}\Vert_{L^{2}_\mathbf{x}L^2_M}\Vert\psi^{12}\Vert_{L^{2}_\mathbf{x}H^1_M}\ds\\
&\leq\,\delta\,\int_0^t\Vert\psi^{12}\Vert^2_{L^{2}_\mathbf{x}H^1_M}\ds +\,c_\delta\,\int_0^t\|\bu^1\|_{W^{s',2}_\mathbf{x}}^2\Vert\psi^{12}\Vert^2_{L^{2}_\mathbf{x}L^2_M}\ds,\\
J_2&\lesssim \int_0^t\|\nabla_x\bu^{12}\|_{L^\infty_{\bx}}\Vert\psi^{2}\Vert_{L^{2}_\mathbf{x}L^2_M}\Vert\psi^{12}\Vert_{L^{2}_\mathbf{x}H^1_M}\ds\\
&\leq\,\delta\,\int_0^t\Vert\psi^{12}\Vert^2_{L^{2}_\mathbf{x}H^1_M}\ds
+\,c_\delta\,\int_0^t\|\bu^{12}\|_{W^{s',2}_\mathbf{x}}^2\Vert\psi^{2}\Vert^2_{L^{2}_\mathbf{x}L^2_M}\ds,\\
J_3,J_5&\lesssim \int_0^t\|\bu^2\|_{W^{s',2}_\mathbf{x}}\Vert\psi^{12}\Vert^2_{L^{2}_\mathbf{x}L^2_M}\ds,\\
J_4,J_6&\lesssim \int_0^t\|\nabx\bu^{12}\|_{L^\infty_{\bx}}\Vert\psi^{1}\Vert_{W^{1,2}_\mathbf{x}L^2_M}\Vert\psi^{12}\Vert_{L^{2}_\mathbf{x}L^2_M}\ds\\
&\lesssim\,\int_0^t\Vert\psi^{12}\Vert^2_{L^{2}_\mathbf{x}L^2_M}\ds+\int_0^t\|\bu^{12}\|_{W^{s',2}_\mathbf{x}}^2\Vert\psi^{1}\Vert^2_{W^{1,2}_\mathbf{x}L^2_M}\ds,
\end{align*}
where $\delta>0$ is arbitrary. Inserting the above into \eqref{distriFormPsiApprZnew0} yields
\begin{align*}
\Vert\psi^{12}(t)&\Vert^2_{L^{2}_\mathbf{x}L^2_M}+2\varepsilon\int_0^t\Vert\psi^{12}\Vert^2_{W^{1,2}_\mathbf{x}L^2_M}\ds
+  \bigg(\frac{A_{11}}{2\lambda} -4\delta \bigg) \int_0^t\Vert\psi^{12} \Vert^2_{L^{2}_\mathbf{x}H^1_M}\ds \\
&\lesssim \int_0^t\|\bu^{12}\|_{W^{s',2}_\mathbf{x}}^2\Vert(\psi^{1},\psi^2)\Vert^2_{W^{1,2}_\mathbf{x}L^2_M}\ds
+
\int_0^t\big(\|(\bu^{1},\bu^2)\|_{W^{s',2}_\mathbf{x}}^2+1\big)\Vert\psi^{12}\Vert^2_{L^{2}_\mathbf{x}L^2_M}\ds.
\end{align*}
Finally, we obtain from Gronwall's lemma
\begin{align}\label{suptx3Znew}
\begin{aligned}
\Vert\psi^{12}(t)&\Vert^2_{L^{2}_\mathbf{x}L^2_M}
+
2\varepsilon\int_0^t\Vert\psi^{12}\Vert^2_{W^{1,2}_\mathbf{x}L^2_M}\ds
+  \bigg(\frac{A_{11}}{2\lambda} -4\delta \bigg) \int_0^t\Vert\psi^{12}\Vert^2_{L^{2}_\mathbf{x}H^1_M}\ds\\
&\leq\,c\exp\bigg(c \int_0^t\|(\bu^{1},\bu^2)\|_{W^{s',2}_\mathbf{x}}^2\ds\bigg) \int_0^t\|\bu^{12}\|_{W^{s',2}_\mathbf{x}}^2\Vert(\psi^{1},\psi^2)\Vert^2_{W^{1,2}_\mathbf{x}L^2_M}\ds.
\end{aligned}
\end{align}
Now we turn to higher order estimates.
Let the multi-index $\alpha$ satisfy
\begin{align}
\label{sPrimesMinus}
\vert \alpha \vert \leq s' \leq s-1
\end{align}
and apply $\partial^\alpha_\bx$ to \eqref{fokkerPlanknew0}  to obtain
\begin{equation}
\begin{aligned}
\label{diffQuoZnew}
&\partial_t \partial^\alpha_\bx \psi^{12}
+ \psi^{12}\, \divx  \partial^\alpha_\bx \bu^2 
+ \psi^1\, \divx  \partial^\alpha_\bx \bu^{12} 
+ \bu^2 \cdot \nabx \partial^\alpha_\bx \psi^{12}
+ \bu^{12} \cdot \nabx \partial^\alpha_\bx \psi^1
\\&=
\varepsilon \Delx \partial^\alpha_\bx\psi^{12}
-
 \divq  \big( (\nabx \partial^\alpha_\bx   \mathbf{u}^1) \mathbf{q} \, \psi^{12} \big) 
-
 \divq  \big( (\nabx \partial^\alpha_\bx   \mathbf{u}^{12}) \mathbf{q} \, \psi^2 \big) 
\\&
+ \frac{A_{11}}{4\lambda}
  \divq  \bigg( M \nabq \partial^\alpha_\bx \bigg(\frac{\psi^{12}}{M} \bigg)
\bigg)
+ K_1^a + K_1^b + K_2^a + K_2^b + \divq K_3^a+ \divq K_3^b
\end{aligned}
\end{equation}
where
\begin{align*}
&K_1^a = \psi^{12}\, \partial^\alpha_\bx \divx \bu^2 - \partial^\alpha_\bx [\psi^{12} \, \divx \bu^2],
\quad
&K_1^b = \psi^1\, \partial^\alpha_\bx \divx \bu^{12} - \partial^\alpha_\bx [\psi^1 \, \divx \bu^{12}],
\\ 
&K_2^a =\bu^2 \cdot \partial^\alpha_\bx  \nabx \psi^{12} - \partial^\alpha_\bx [\bu^2 \cdot \nabx \psi^{12}],
\quad
&K_2^b =\bu^{12} \cdot \partial^\alpha_\bx  \nabx \psi^1 - \partial^\alpha_\bx [\bu^{12} \cdot \nabx \psi^1],
\\
&K_3^a = (\partial^\alpha_\bx  \nabx \bu^1) \bq\, \psi^{12} - \partial^\alpha_\bx  [(\nabx \bu^1) \bq \, \psi^{12}],
\quad
&K_3^b = (\partial^\alpha_\bx  \nabx \bu^{12}) \bq\, \psi^2 - \partial^\alpha_\bx  [(\nabx \bu^{12}) \bq \, \psi^2].
\end{align*}
Testing \eqref{diffQuoZnew} with $\partial^\alpha_\bx \frac{\psi^{12}}{M}$ and integrating over $\Q_t \times B$ yields
\begin{equation}
\begin{aligned}
\label{distriFormPsiApprZnew}
&\frac{1}{2}\int_{\tor \times B} M  \bigg\vert \partial^\alpha_\bx\frac{ \psi^{12}(t)}{M} \bigg\vert^2 \dq \dx  
+
\varepsilon
\int_{\Q_t \times B} M \bigg\vert \nabx\partial^\alpha_\bx \frac{\psi^{12}}{M} \bigg\vert^2 \dq  \dx \, \mathrm{d}\sigma
\\&  
+
\frac{A_{11}}{4\lambda}
\int_{\Q_t \times B} M \bigg\vert \nabq \partial^\alpha_\bx \frac{\psi^{12}}{M} \bigg\vert^2 \dq  \dx\, \mathrm{d}\sigma  
=
\int_{\Q_t \times B}M (\nabx \partial^\alpha_\bx   \bu^1 )\mathbf{q} \frac{\psi^{12}}{M} \cdot \nabq  \partial^\alpha_\bx \frac{\psi^{12}}{M} \dq \dx\, \mathrm{d}\sigma
\\&  
+
\int_{\Q_t \times B}M (\nabx \partial^\alpha_\bx   \bu^{12} )\mathbf{q} \frac{\psi^2}{M} \cdot \nabq  \partial^\alpha_\bx \frac{\psi^{12}}{M} \dq \dx\, \mathrm{d}\sigma
+
\int_{\Q_t \times B}\frac{1}{2} M (\divx \bu^2 ) \bigg\vert  \partial^\alpha_\bx\frac{ \psi^{12}}{M} \bigg\vert^2 \dq \dx\, \mathrm{d}\sigma 
\\&
-
\int_{\Q_t \times B} M  \bu^{12}\cdot \nabx    \partial^\alpha_\bx\frac{ \psi^1}{M}\cdot    \partial^\alpha_\bx\frac{ \psi^{12}}{M}  \dq \dx \, \mathrm{d}\sigma  
-
\int_{\Q_t \times B}  \psi^{12} \,\divx  \partial^\alpha_\bx \bu^2   \cdot \partial^\alpha_\bx \frac{\psi^{12}}{M} \dq \dx \, \mathrm{d}\sigma  
\\&
-
\int_{\Q_t \times B}  \psi^1 \,\divx  \partial^\alpha_\bx \bu^{12}   \cdot \partial^\alpha_\bx \frac{\psi^{12}}{M} \dq \dx \, \mathrm{d}\sigma
+
\int_{\Q_t \times B}
(K_1^a+K_1^b+K_2^a+K_2^b) \cdot \partial^\alpha_\bx \frac{\psi^{12}}{M} 
\dq \dx\, \mathrm{d}\sigma  
\\
&
-
\int_{\Q_t \times B}
(K_3^a+K_3^b)\cdot \nabq  \partial^\alpha_\bx \frac{\psi^{12}}{M} 
\dq \dx \, \mathrm{d}\sigma,
\end{aligned}
\end{equation}
recall \eqref{fokkerPlankIintialnew0}. Now note that
\begin{equation}
\begin{aligned}
\bigg\vert&
\int_{\Q_t \times B} M  \bu^{12}\cdot \nabx    \partial^\alpha_\bx\frac{ \psi^1}{M}\cdot    \partial^\alpha_\bx\frac{ \psi^{12}}{M}  \dq \dx\, \mathrm{d}\sigma  
\bigg\vert\\
&\lesssim
\int_{\Q_t \times B} M \bigg\vert \partial^\alpha_\bx \frac{\psi^{12}}{M} \bigg\vert^2 \dq \dx\, \mathrm{d}\sigma
+\int_0^t \big\Vert \bu^{12} \big\Vert_{L^\infty_\bx}^2
\int_{\tor \times B} M \bigg\vert \nabx \partial^\alpha_\bx \frac{\psi^1}{M} \bigg\vert^2 \dq \dx\, \mathrm{d}\sigma\\
&\lesssim \int_0^t
\Vert\psi^{12}\Vert^2_{W^{s',2}_\mathbf{x}L^2_M}\ds+\int_0^t\Vert \bu^{12}  \Vert^2_{W^{s',2}_\bx}
\Vert\psi^1\Vert^2_{W^{s'+1,2}_\mathbf{x}L^2_M}\ds
\end{aligned}
\end{equation}
whereas the rest of the right-hand terms of \eqref{distriFormPsiApprZnew} can be estimated exactly as in \eqref{Residual1aZ}--\eqref{Residual12aZ}. In analogy with \eqref{suptxZ}, we obtain
\begin{align}\label{suptxZnew} 
\begin{aligned}
\Vert\psi^{12} &(t)\Vert^2_{W^{s',2}_{\bx}L^2_M}
+2\varepsilon 
\int_0^t\Vert\psi^{12}\Vert^2_{W^{s'+1,2}_{\bx}L^2_M}\, \mathrm{d}\sigma
+ \bigg(\frac{A_{11}}{2\lambda} -8\delta \bigg)
\int_0^t\Vert\psi^{12}\Vert^2_{W^{s',2}_{\bx}H^1_M}\ds 
\\&
\lesssim \int_0^t\big( \Vert(\bu^{1},\bu^2)\Vert^2_{W^{s'+1,2}_\bx}+1\big)\Vert\psi^{12}\Vert^2_{W^{s',2}_{\bx}L^2_M}\ds+\int_0^t \Vert\bu^{12}\Vert^2_{W^{s'+1,2}_\bx}\Vert(\psi^1,\psi^2)\Vert^2_{W^{s'+1,2}_{\bx}L^2_M}\ds\\
&+
\int_{\Q_t \times B}
\frac{1}{M}\big( \vert K_1^a\vert^2 +\vert K_1^b\vert^2 +\vert K_2^a\vert^2+\vert K_2^b\vert^2+\vert K_3^a\vert^2+\vert K_3^b\vert^2 \big) 
\dq \dx\, \mathrm{d}\sigma
\end{aligned}
\end{align}
where $c=c(b,d)$. 

The commutator-terms can be estimated analogously to \eqref{suptxKZ} such that
\begin{equation}
\begin{aligned}
\label{suptxKZnew}
&\int_{\Q_t \times B}
\frac{1}{M} \big( \vert K_1^a\vert^2 +\vert K_1^b\vert^2 +\vert K_2^a\vert^2+\vert K_2^b\vert^2+\vert K_3^a\vert^2+\vert K_3^b\vert^2 \big)  
\dq \dx\, \mathrm{d}\sigma
\\ &\lesssim 
\int_0^t \Big(\Vert \nabx  \bu^1 \Vert_{L^\infty_\bx}^2  +\Vert \nabx  \bu^2 \Vert_{L^\infty_\bx}^2 + \Vert  \nabx^{s'}  \bu^1  \Vert^2_{L^2_\bx} + \Vert \nabx^{s'}  \bu^2  \Vert^2_{L^2_\bx}\Big) 
\int_{\tor\times B}M\bigg\vert \nabx^{s'} \frac{  \psi^{12}  }{M} \bigg\vert^2 
\dq \dx\, \mathrm{d}\sigma
\\&
+ 
\int_0^t \Big( \Vert \nabx  \bu^{12} \Vert_{L^\infty_\bx}^2 +\Vert  \nabx^{s'}  \bu^{12}  \Vert^2_{L^2_\bx}\Big) 
\int_{\tor\times B}
\bigg( M\bigg\vert \nabx^{s'} \frac{  \psi^1  }{M} \bigg\vert^2
+ 
M\bigg\vert \nabx^{s'} \frac{  \psi^2  }{M} \bigg\vert^2 \bigg)
\dq \dx\, \mathrm{d}\sigma\\
&\lesssim \int_0^t\Vert (\bu^{1},\bu^2)  \Vert^2_{W^{s',2}_\bx}
\Vert\psi^{12}\Vert^2_{W^{s',2}_\mathbf{x}L^2_M}\ds+\int_0^t\Vert \bu^{12}  \Vert^2_{W^{s',2}_\bx}
\Vert(\psi^1,\psi^2)\Vert^2_{W^{s',2}_\mathbf{x}L^2_M}\ds
\end{aligned}
\end{equation}
where the hidden constant is $c=c(s,b,d)$. Substituting \eqref{suptxKZnew} into \eqref{suptxZnew} yields
\begin{align}\label{suptxZaanew}
\begin{aligned}
\Vert\psi^{12} &(t)\Vert^2_{W^{s',2}_{\bx}L^2_M}
+2\varepsilon 
\int_0^t\Vert\psi^{12}\Vert^2_{W^{s'+1,2}_{\bx}L^2_M}\, \mathrm{d}\sigma
+  \bigg(\frac{A_{11}}{2\lambda} -8\delta \bigg)
\int_0^t\Vert\psi^{12}\Vert^2_{W^{s',2}_{\bx}H^1_M}\ds
\\&
\lesssim \int_0^t\big( \Vert(\bu^{1},\bu^2)\Vert^2_{W^{s'+1,2}_\bx}+1\big)\Vert\psi^{12}\Vert^2_{W^{s',2}_{\bx}L^2_M}\ds+\int_0^t \Vert\bu^{12}\Vert^2_{W^{s'+1,2}_\bx}\Vert(\psi^1,\psi^2)\Vert^2_{W^{s'+1,2}_{\bx}L^2_M}\ds.
\end{aligned}
\end{align}
As a consequence of Gronwall's lemma, we obtain
\begin{align}\label{f4t}
\begin{aligned}
\sup_{0<t<T}&\Vert\psi^1-\psi^2\Vert^2_{W^{s',2}_\mathbf{x}L^2_M}+2\varepsilon\int_0^T\Vert\psi^1-\psi^2\Vert^2_{W^{s'+1,2}_\mathbf{x}L^2_M}\,\mathrm{d}t
+ \bigg(\frac{A_{11}}{2\lambda} -8\delta \bigg)
\int_0^T\Vert\psi^1-\psi^2\Vert^2_{W^{s',2}_\mathbf{x}H^1_M}\,\mathrm{d}t
\\&
\leq\,c
\exp\bigg( c\int_0^T \Vert (\mathbf{u}^1,\mathbf{u}^2)   \Vert_{W^{s'+1,2}_\mathbf{x}}^2
\, \mathrm{d}t\bigg)\int_0^T\Vert \bu^{1}-\bu^2  \Vert^2_{W^{s'+1,2}_\bx}
\Vert(\psi^1,\psi^2)\Vert^2_{W^{s'+1,2}_\mathbf{x}L^2_M} \mathrm{d}t
\end{aligned}
\end{align}
which completes the proof of Theorem \ref{thm:main'FP}.

\section{The coupled system}
\label{sec:coupled}
After solving the fluid system and the Fokker-Planck equation both independently from each other in the two previous sections, we are now in the position to solve the coupled system.
This shall be done by a fixed point argument which finally leads to the proof of the main result from Theorem \ref{thm:aux}. Set 
$$X^s= L^\infty\big(0,T;  W^{s,2} (\tor; L^2_M(B)) \big)
 \cap L^2\big(0,T; W^{s,2} (\tor; H^1_M(B)) \big)$$
 which is a Banach space equipped with the norm $$\|\cdot\|_{X^s}^2=\sup_{0<t<T}\|\cdot\|_{W^{s,2}_\bx L^2_M}^2+\int_0^T\|\cdot\|_{W^{s,2}_\bx H^1_M}^2\dt.$$
For $\widetilde{\psi} \in X^s$, let $(r,\bu)$ be the unique solution to \eqref{E3}--\eqref{E4} with data $(r_0,\bu_0,\bff,\mathbb T(\widetilde{\psi}))$. The existence of such a solution of class
\begin{align*}
r &\in  C([0,T]; W^{s,2}(\tor)), \ r>0,\quad
\vu\in C([0,T]; W^{s,2}(\tor;\mathbb{R}^d))\cap L^2(0,T;W^{s+1,2}(\mt;\mathbb{R}^d)),
\end{align*}
is guaranteed by Theorems \ref{thm:main} and \ref{thm:main'} noticing that
$\widetilde{\psi} \in X^s$ implies $\mathbb T(\widetilde{\psi})\in C([0,T]; W^{s,2}(\tor; \mathbb{R}^{d \times d}))$ by Lemma \ref{lem:A1}.
Now, given $(r,\bu)$ with the regularity above, we can solve equation \eqref{eq:FP}
using Theorems \ref{thm:mainFP} and \ref{thm:main'FP}. Hence we obtain
a unique $\psi\in X^s$. We denote the mapping $\widetilde{\psi} \mapsto \psi$ by $\mathfrak T$.
We start with the following lemma.
\begin{lemma}\label{lems:1}
Let $s>\frac{d}{2}+2$. There is $T>0$ and $\mathfrak R>0$ such that $\mathfrak T:\mathfrak B^s_{\mathfrak R}\rightarrow \mathfrak B^s_{\mathfrak R}$,
where $\mathfrak B_{\mathfrak R}^s$ is the closed ball in $X^s$ with radius $\mathfrak R$.
\end{lemma}
\begin{proof}
Choosing $T$ sufficiently small we obtain from Theorem \ref{thm:main} that
\begin{align*}
\sup_{0<t<T}  &\big\Vert (r , \mathbf{u}   ) \big\Vert^2_{W^{s,2}_\mathbf{x}}
+ \int_{0}^T\big\Vert \mathbf{u}  \big\Vert^2_{W^{s+1,2}_\mathbf{x}}
\dt \leq c(r_0,\bu_0,\bff)+
c\,\int_0^T
\Vert   \mathbb T(\widetilde{\psi})\Vert_{W^{s,2}_\mathbf{x}}^2
 \, \mathrm{d}t.
\end{align*}
We infer further
for arbitrary $\delta>0$
\begin{align*}
\sup_{0<t<T}  \big\Vert (r , \mathbf{u}   ) \big\Vert^2_{W^{s,2}_\mathbf{x}}
+ \int_{0}^T\big\Vert \mathbf{u}  \big\Vert^2_{W^{s+1,2}_\mathbf{x}}
\dt &\leq c+
Tc_\delta\sup_{0<t<T}\Vert   \widetilde{\psi} \Vert_{W^{s,2}_\mathbf{x}L^2_M}^2+\delta\int_0^T\Vert   \widetilde{\psi} \Vert_{W^{s,2}_\mathbf{x} H^1_M}^2
 \, \mathrm{d}t\\
 &\leq \,c+(Tc_\delta+\delta)\mathfrak R^2
\end{align*}
using Lemma \ref{lem:A1}. Now, we first choose $\delta=\mathfrak R^{-2}$ and then $T$ so small such 
$Tc_\delta\leq 1$. Consequently, the right-hand side is bounded by a constant only depending on the given data. On the other hand, we obtain from Theorem \ref{thm:mainFP}
\begin{align*}
\sup_{0<t<T}\Vert\psi\Vert^2_{W^{s,2}_\mathbf{x}L^2_M}
+ \int_0^T\Vert\psi\Vert^2_{W^{s,2}_\mathbf{x}H^1_M}\,\mathrm{d}t
&\leq  c\exp\bigg(c\int_0^T\Vert   \bu  \Vert^2_{W^{s+1,2}_\bx}\,\mathrm{d}t \bigg) \Vert \psi_0\Vert^2_{W^{s,2}_\mathbf{x}L^2_M}\leq \,c
\end{align*}
where the constant only depends on the data (but we have to choose $T$ small enough). It is bounded by 
$\mathfrak R^2$ provided the latter one was chosen  large enough compared to the data $(r_0,\bu_0,\bff,\psi_0)$.
\end{proof}
As a by-product  of the proof of Lemma \ref{lems:1}, we obtain the following corollary.
\begin{corollary}\label{cor:main}
Let $(r,\bu)$ be the unique solution to \eqref{E3}--\eqref{E4} with data $(r_0,\bu_0,\bff,\mathbb T(\tilde\psi))$. Under the assumptions of Lemma \ref{lems:1} we have
\begin{align*}
\sup_{0<t<T}  \big\Vert (r , \mathbf{u}   ) \big\Vert^2_{W^{s,2}_\mathbf{x}}
+ \int_{0}^T\big\Vert \mathbf{u}  \big\Vert^2_{W^{s+1,2}_\mathbf{x}}
\dt\leq \,c
\end{align*}
where $c$ only depends on the data $(r_0,\bu_0,\bff)$.
\end{corollary}
In the next step, we have to show that $\mathfrak T$ is a contraction. Unfortunately, we are unable to show this on $X^s$.  However, 
we can consider a ball in the smaller space $X^s$ equipped with the topology of the larger space $X^{s'}$, where $s'\leq s-1$.
We have the following result.
\begin{lemma}\label{lems:2}
Let $s>\frac{d}{2}+2$. There is $T>0$
 and $\mathfrak R>0$ 
such that $\mathfrak T:\mathfrak B_{\mathfrak R}^s\rightarrow \mathfrak B_{\mathfrak R}^s$ is a contraction on $X^{s'}$.
\end{lemma} 
\begin{proof}
Let $(r^1,\bu^1)$ and $(r^2,\bu^2)$ be two strong solutions to \eqref{E3}--\eqref{E4}
with data $(r_0,\vu_0,\bff,\mathbb T(\widetilde{\psi}^1))$ and $(r_0,\vu_0,\bff,\mathbb T(\widetilde{\psi}^2))$ respectively.
We obtain from Theorem \ref{thm:main'}
\begin{align*}
\int_0^T \Vert \mathbf{u}^{1}-\bu^2   \Vert^2_{W^{s'+1,2}_\mathbf{x}}\, \mathrm{d}t
&
\leq\,c
\exp\bigg(c\int_0^T \Big( 
\Vert (r^1,r^2)  \Vert^2_{W^{s'+1,2}_\mathbf{x}}
+ \Vert (\mathbf{u}^1,\mathbf{u}^2)   \Vert_{W^{s'+2,2}_\mathbf{x}}
+
\Vert   \mathbb T(\widetilde{\psi}^2)\Vert_{W^{s',2}_\mathbf{x}}^2+1\Big)
\, \mathrm{d}t \bigg)\\
&\qquad\qquad\times
\int_0^T\Vert   \mathbb T(\widetilde{\psi}^1)-\mathbb T(\widetilde{\psi}^2)\Vert_{W^{s',2}_\mathbf{x}}^2
 \Vert r^1 \Vert_{W^{s'+1,2}_\mathbf{x}}
\, \mathrm{d}t\\
&\leq\,c(\mathfrak R)\int_0^T\Vert   \mathbb T(\widetilde{\psi}^1)-\mathbb T(\widetilde{\psi}^2)\Vert_{W^{s',2}_\mathbf{x}}^2
\, \mathrm{d}t 
\end{align*}
using also Lemma \ref{lems:1} (together with Lemma \ref{lem:A1}) and Corollary \ref{cor:main}. 
On the other hand, let $\psi^1$ and $\psi^2$ be two  solutions to \eqref{eq:FP}
with data $(\psi_0,\bu^1)$ and $(\psi_0,\bu^2)$ respectively. Then Theorem \ref{thm:main'FP} tells us that
\begin{align*}
\sup_{0<t<T}&\Vert\psi^1-\psi^2\Vert^2_{W^{s',2}_\mathbf{x}L^2_M}
+
\int_0^T \Vert\psi^1-\psi^2\Vert^2_{W^{s',2}_\mathbf{x}H^1_M} \mathrm{d}t
\\&
\leq\,c
\exp\bigg(c\int_0^T \Vert (\mathbf{u}^1,\mathbf{u}^2)   \Vert_{W^{s'+1,2}_\mathbf{x}}^2
\, \mathrm{d}t \bigg)\int_0^T\Vert \bu^{1}-\bu^2  \Vert^2_{W^{s'+1,2}_\bx}
\Vert(\psi^1,\psi^2)\Vert^2_{W^{s'+1,2}_\mathbf{x}L^2_M} \mathrm{d}t \\
&\leq\,c(\mathfrak R)\int_0^T\Vert \bu^{1}-\bu^2  \Vert^2_{W^{s'+1,2}_\bx}\mathrm{d}t
\end{align*}
using again Lemma \ref{lems:1} and Corollary \ref{cor:main}. Combining both shows that
\begin{align*}
\sup_{0<t<T}&\Vert\psi^1-\psi^2\Vert^2_{W^{s',2}_\mathbf{x}L^2_M}
+
\int_0^T \Vert\psi^1-\psi^2\Vert^2_{W^{s',2}_\mathbf{x}H^1_M} \mathrm{d}t
\leq\,c(\mathfrak R)\int_0^T\Vert   \mathbb T(\widetilde{\psi}^1)-\mathbb T(\widetilde{\psi}^2)\Vert_{W^{s',2}_\mathbf{x}}^2
\, \mathrm{d}t.
\end{align*}
Finally, we infer from Lemma \ref{lem:A1} (with a suitable choice of $\delta$)
\begin{align*}
\sup_{0<t<T}&\Vert\psi^1-\psi^2\Vert^2_{W^{s',2}_\mathbf{x}L^2_M}
+
\int_0^T \Vert\psi^1-\psi^2\Vert^2_{W^{s',2}_\mathbf{x}H^1_M} \mathrm{d}t\\
&\leq\,c(\mathfrak R)T\sup_{0<t<T}\Vert    \widetilde{\psi} ^1 - \widetilde{\psi}^2\Vert_{W^{s',2}_\mathbf{x}L^2_M}^2
+\frac{1}{2}\int_0^T\Vert   \widetilde{\psi}^1 - \widetilde{\psi}^2\Vert_{W^{s',2}_\mathbf{x}H^1_M}^2
\, \mathrm{d}t.
\end{align*}
The claim follows provided we choose $T$ small enough to guarantee $c(\mathfrak R)T\leq\frac{1}{2}$.
\end{proof}

\begin{proof}[Proof of Theorem \ref{thm:aux}]
The claim of Theorem \ref{thm:aux} follows from Lemma \ref{lems:2} due to Banach's fixed point theorem.
\end{proof}

\begin{proof}[Proof of Theorem \ref{thm:main0}]
We have proved Theorem \ref{thm:aux}.
Assume that the data $(\varrho_0, \mathbf{u}_0, \psi_0, \mathbf{f})$ satisfy the hypothesis of Theorem \ref{thm:main0}. Setting $r_0:=\sqrt{\frac{2a\gamma}{\gamma -1}} \varrho_0^{\frac{\gamma-1}{2}}$ we see that the initial data $(r_0, \mathbf{u}_0, \psi_0, \mathbf{f})$ satisfy the assumptions
of Theorem \ref{thm:aux} (in particular, $r_0$ is strictly positive). We obtain a unique solution $(r,\bfu,\psi)$ to \eqref{contEquR}--\eqref{fokkerPlankR} with positive $r$.
Using the transformation 
\begin{align*}
\varrho :=\bigg( \frac{\gamma-1}{2a\gamma} \bigg)^\frac{1}{\gamma -1 } r^\frac{2}{\gamma-1}
\end{align*}
it is now straightforward to see that $(\varrho,\bfu,\psi)$ is the unique  solution  to \eqref{contEq}--\eqref{fokkerPlank} defined in the same existence interval.
\end{proof}

A consequence of the proof of Theorem \ref{thm:main} is the result below. It gives a blowup criterion for the Cauchy problem of the Navier--Stokes--Fokker--Planck system.
\begin{corollary}
\label{cor:blowup}
Let $(\varrho, \mathbf{u}, \psi)$ be a solution to problem \eqref{contEq}--\eqref{fokkerPlank}
in the sense of Definition \ref{def:strsolmartRho} on the interval $[0,T_{\max})$ with the data $(\varrho_0, \mathbf{u}_0, \psi_0, \mathbf{f})$.
If $T_{\max}<\infty$ is the maximal existence time, then
\begin{align*}
\limsup_{T \rightarrow T_{\max}}\bigg[
\Vert \bu \Vert_{W^{2,\infty}_\bx}
+
\Vert \mathrm{div}_\bx \mathbb{T}(\psi) \Vert_{L^\infty_\bx} \bigg]=\infty.
\end{align*} 
\end{corollary}
\begin{proof}
Assume that the preamble and assumption of Corollary \ref{cor:blowup} holds true but that the conclusion is false. Then 
\begin{align*}
\limsup_{T \rightarrow T_{\max}}\bigg[
\Vert \bu \Vert_{W^{2,\infty}_\bx}
+
\Vert \mathrm{div}_\bx \mathbb{T}(\psi) \Vert_{L^\infty_\bx} \bigg]\lesssim R.
\end{align*} 
for some $R<\infty$ which in turn yield
\begin{align*}
\varrho(t,\bx) \leq  c(R) \quad \text{for all}\quad (t,\bx)\in [0,T_{\max}] \times \tor
\end{align*}
by virtue of the maximum principle for \eqref{contEq}. By the transformation \eqref{transfRhoToR}, it follows that \eqref{lipschitzDR} still hold and thus, validating the a priori estimates established in Sections \ref{sec:exisNS} and \ref{sec:exisFokkerPlanck}   on $[0,T_{\max}] $.
\\
The equivalence of \eqref{contEq}--\eqref{momEq} and \eqref{contEquR}--\eqref{momEquR} and the fact that $(\varrho, \mathbf{u}, \psi)$ satisfies the a priori estimates means that 
\begin{equation}
\begin{aligned}
&\varrho \in C \big([0,T_{\max}] ;  W^{1, s}(\tor) \big), 
\\ 
 &\mathbf{u} \in  C\big([0,T_{\max}] ;  W^{s,2} (\tor) \big)
 \cap L^2\big(0,T_{\max} ;  W^{s+1,2} (\tor)\big),
 \\ 
 &\psi \in  C\big([0,T_{\max}]; W^{s,2} (\tor; L^2_M(B)) \big)
\cap L^2\big(0,T_{\max} ; W^{s,2} (\tor; H^1_M(B)) \big).
\end{aligned}
\end{equation}
Now note that the density remains positive for a positive initial density. It means that we can take $(\varrho, \bu, \psi, \mathbf{f})$ as a new data so that   Theorem \ref{thm:main0} establishes the existence of a time $T^{\max}_{\max}>T_{\max}$ on which a solution exist. This contradicts the fact that $T_{\max}$ is maximal.\\\
\end{proof}

\centerline{\bf Acknowledgement}
\noindent{
The authors would like to thank the referee for the careful reading of the manuscript and the valuable suggestions.}\\\

%

\end{document}